\documentclass[11pt,reqno]{amsart}

\usepackage[margin=1in]{geometry}
\usepackage{amsmath,amssymb,amsthm,mathtools}
\usepackage{aliascnt}
\numberwithin{equation}{section}
\usepackage{enumitem}
\usepackage{microtype}
\usepackage{hyperref}
\usepackage[nameinlink,capitalize,noabbrev]{cleveref}
\usepackage{cite}

\allowdisplaybreaks

\hypersetup{
  colorlinks=true,
  linkcolor=blue,
  citecolor=blue,
  urlcolor=blue
}

\newtheorem{theorem}{Theorem}[section]
\newaliascnt{proposition}{theorem}
\newtheorem{proposition}[proposition]{Proposition}
\aliascntresetthe{proposition}
\newaliascnt{lemma}{theorem}
\newtheorem{lemma}[lemma]{Lemma}
\aliascntresetthe{lemma}
\newaliascnt{corollary}{theorem}
\newtheorem{corollary}[corollary]{Corollary}
\aliascntresetthe{corollary}

\crefname{appendix}{Appendix}{Appendices}
\Crefname{appendix}{Appendix}{Appendices}

\theoremstyle{remark}
\newaliascnt{remark}{theorem}
\newtheorem{remark}[remark]{Remark}
\aliascntresetthe{remark}

\newcommand{\X}{\mathcal X}
\newcommand{\F}{\mathcal F}
\newcommand{\G}{\mathcal G}
\newcommand{\B}{\mathcal B}
\newcommand{\Pp}{\mathbb P}
\newcommand{\Ep}{\mathbb E}
\newcommand{\R}{\mathbb R}
\newcommand{\1}{\mathbf 1}
\newcommand{\ip}[2]{\langle #1,#2\rangle}
\newcommand{\norm}[1]{\lVert #1\rVert}
\newcommand{\normtwo}[1]{\lVert #1\rVert_2}
\newcommand{\normg}[1]{\lVert #1\rVert_{L^2(g)}}
\newcommand{\diam}{\operatorname{diam}_2}
\newcommand{\dist}{\operatorname{dist}_2}
\newcommand{\clip}{\operatorname{clip}_{[0,1]}}
\newcommand{\eps}{\varepsilon}
\newcommand{\TV}{\operatorname{TV}}

\title[Adaptive confidence sets without design smoothness]{Adaptive Confidence Sets for Binary Regression without Design Smoothness}
\author{P. M. Aronow and Patrick Lopatto}
\date{}

\begin{document}

\begin{abstract}
We study honest adaptive confidence sets for the regression function in
random-design binary regression under \(L^2(dx)\) loss. Assuming only known
bounds \(0<c\leq g\leq C<\infty\) on the unknown design density, we construct
asymptotically honest, rate-adaptive confidence sets without requiring $g$ to be smooth. Full adaptation is possible when the range of
regression-function smoothness spans at most a factor of two. Over wider
smoothness ranges, adaptation is achieved on the usual separated classes at
the corresponding testing rates \(n^{-2s/(4s+d)}\). A lower bound under the
uniform design shows that these separation rates are rate-optimal. This
answers a question raised by Mukherjee and Sen
\cite{MukherjeeSen2018}.
\end{abstract}

\maketitle

\section{Introduction}

\subsection{Motivation and overview}

Adaptive confidence sets address a basic tension in nonparametric inference: a set calibrated to the roughest admissible functions may be too large to be informative, while a smaller set tailored to smoother functions may lose coverage. The central question is whether a single procedure can retain uniform validity while shrinking at a rate that reflects the unknown smoothness. The minimax theory of adaptive confidence sets determines when these goals are compatible and identifies the relevant statistical barriers when they are not.

We are interested here in adaptive confidence sets for nonparametric regression problems. Although the construction we present readily generalizes, to simplify the exposition, we state our results in the specific case of random-design binary regression.  We observe
independent copies of a pair $(X,Y)$, where $X\in[0,1]^d$ and $Y\in\{0,1\}$,
with
\begin{equation}\label{eq:intro-model}
 X\sim g,
 \qquad
 \Pp(Y=1\mid X=x)=f(x).
\end{equation}
The regression function $f$ is the conditional success probability and is the
object of interest; the design density $g$ is an unknown nuisance parameter.
We measure uncertainty in the unweighted global norm
\begin{equation}\label{eq:intro-L2}
 \normtwo{h}^2=\int_{[0,1]^d}h(x)^2\,dx.
\end{equation}
Unlike the design-weighted norm $L^2(g)$, this loss does not change with the
distribution of the covariates and does not downweight regions merely because
they are sampled less often.  This makes the target natural when the regression
function itself, rather than only its behavior under one particular design, is
of interest.  It also creates a central nuisance difficulty: the data arrive
according to $g$, while the confidence set is evaluated with respect to
Lebesgue measure.

Suppose that $f$ belongs to a Besov ball of unknown smoothness
$\beta\in[\beta_-,\beta_+]$. By analogy with Stone's classical minimax
theory for random-design nonparametric regression \cite{Stone1982}, one expects that, when $\beta$ is known, the squared $L^2(dx)$ estimation
risk has the characteristic scale $\eps_n(\beta)^2$, where
\begin{equation}\label{eq:intro-estimation-rate}
 \eps_n(\beta)=n^{-\beta/(2\beta+d)}.
\end{equation}
In Gaussian nonparametric models, classical adaptive-estimation theory shows
 that the corresponding minimax rates can be attained without
knowing the smoothness index
\cite{Lepskii1991,DonohoJohnstone1995}. Related adaptive-estimation results
for random-design regression, including binary-response models, are available
under additional regularity assumptions on the design distribution
\cite{AntoniadisLeblanc2000,Baraud2002,MukherjeeSen2018}.

Honest adaptive uncertainty quantification is more restrictive. For an $s$-smooth class, the relevant
testing radius is
\begin{equation}\label{eq:intro-testing-rate}
 \rho_n(s)=n^{-2s/(4s+d)},
\end{equation}
and the identity
\begin{equation}\label{eq:intro-rate-identity}
 \rho_n(s)=\eps_n(2s)
\end{equation}
underlies the familiar factor-of-two smoothness boundary.  Over a smoothness range
contained in $[s,2s]$, honest adaptation is possible on the full parameter
space.  Over a wider range, rough functions lying within order $\rho_n(s)$ of
a smoother class cannot be reliably distinguished from that class and must be
excluded if one insists on the smaller smooth-class diameter.  This is the
classical regression-function testing barrier.  The question addressed here is
whether an unknown and potentially very rough design density creates an 
additional barrier.

\subsection{The unknown-design obstacle and our contribution}

The unknown design appears, at first, to create a serious extra difficulty.
If $a$ is a preliminary estimator of $f$ and $r=f-a$, then the regression
identity $\Ep(Y\mid X)=f(X)$ gives observable moments of the form
\[
 \Ep\{(Y-a(X))\varphi(X)\}
 =\int r(x)\varphi(x)g(x)\,dx.
\]
Thus the data naturally reveal the design-weighted residual $rg$, whereas the
radius of the desired confidence set must control the unweighted residual $r$.
When $g$ is smooth, one may try to estimate it and remove the weighting.  When
$g$ has jumps or oscillations at arbitrarily fine scales, however,
multiplication by $g$ can transfer energy between wavelet scales, so a direct
coefficientwise comparison between $rg$ and $r$ is unavailable.

This issue is pivotal in the work of Mukherjee and Sen
\cite{MukherjeeSen2018}, who developed adaptive estimators, goodness-of-fit
tests, and honest adaptive confidence sets for binary regression with unknown
random design.  Their confidence-set theorem estimates the nuisance density
and assumes that its lower smoothness index $\gamma_-$ satisfies
$2\beta_+<\gamma_-$.  This left open whether substantial smoothness of $g$ is
intrinsically necessary for honest adaptation under $L^2(dx)$ loss, or whether
it is a consequence of estimating and dividing by the design density.

We show that this design-smoothness requirement is not intrinsic.  Assume only known constants
$0<c\leq1\leq C<\infty$ such that
\begin{equation}\label{eq:intro-overlap}
 0<c\leq g(x)\leq C<\infty
 \quad\text{for almost every }x.
\end{equation}
Uniformly over this full class of design densities, we construct
asymptotically honest confidence sets that attain the minimax estimation
diameter throughout every factor-of-two smoothness range.  Over wider ranges,
the same diameters are attained on the usual separated classes, with separation
at the testing rate \eqref{eq:intro-testing-rate}.  A lower bound under the
uniform design shows that this separation rate is sharp in order.  Moreover, the design class always contains densities that belong to no
Besov space of positive smoothness (when $c<1<C$).  Thus, under fixed overlap, arbitrarily
rough design densities do not worsen the attainable confidence-set diameter or
the required separation rate.  This answers the question raised by Mukherjee
and Sen \cite{MukherjeeSen2018}.

The principal new contribution is an upper-bound mechanism that removes the
design-smoothness assumption. We show that, under fixed upper and lower bounds on $g$, the roughness of the design density affects neither the attainable confidence-set diameter nor the necessary separation rate. Their optimal orders remain governed solely by the classical regression-function estimation and testing barriers.
Moreover, this mechanism is not specific to binary responses and
extends to bounded responses with conditional mean $f$, as discussed in
\Cref{rem:bounded-responses} below.  The factor-of-two adaptation regimes,
separated-space formulation, and lower-bound strategy come from the existing
theory of adaptive confidence sets.

\subsection{The key idea: finite-dimensional coercivity}

The argument avoids estimating either $g$ or $1/g$.  Let $V_J$ be a
finite-dimensional wavelet approximation space, let $P_J$ be the orthogonal
projection onto $V_J$, and let
$(\varphi_{J,\lambda})_{\lambda\in\Lambda_J}$ be its canonical orthonormal
basis.  Given an adaptive pilot estimator $\widetilde f$, we use the projected center
$a=P_J\widetilde f\in V_J$.  Writing
\[
 \Phi_J(x)=\bigl(\varphi_{J,\lambda}(x)\bigr)_{\lambda\in\Lambda_J}
 \in\mathbb R^{D_J},
 \qquad
 D_J=|\Lambda_J|=\dim(V_J),
\]
the residual-moment vector associated with this center is
\begin{equation}\label{eq:intro-residual-moment}
\begin{aligned}
 \mu_J(a)
 &=\Ep\bigl[(Y-a(X))\Phi_J(X)\bigr] \\
 &=\left(
  \int_{\X}(f-a)(x)\varphi_{J,\lambda}(x)g(x)\,dx
 \right)_{\lambda\in\Lambda_J}.
\end{aligned}
\end{equation}
Empirical versions of $\mu_J(a)$ can be formed on two independent data blocks,
and, conditional on $a$, their inner product is unbiased for
\begin{equation}\label{eq:intro-wrong-quantity}
 \norm{\mu_J(a)}_{\ell^2}^2
 =\normtwo{P_J((f-a)g)}^2.
\end{equation}
The issue is therefore to turn control of the quantity on the right into
control of $\normtwo{f-a}$ without requiring any regularity of $g$.

The key observation is as follows.  Let
$M_g$ denote multiplication by $g$.  For every $v\in V_J$,
\begin{equation}\label{eq:intro-coercivity}
 \ip{P_J(gv)}{v}
 =\int_\X g(x)v(x)^2\,dx
 \geq c\normtwo v^2.
\end{equation}
Hence the compressed multiplication operator $P_JM_gP_J$ is uniformly
positive and invertible on $V_J$, even when $g$ is discontinuous or oscillates
at every scale.  Decomposing a residual $r$ as
$r=P_Jr+(I-P_J)r$ and using also $g\leq C$ gives the deterministic one-sided
bound
\begin{equation}\label{eq:intro-bound}
 \normtwo r^2
 \lesssim
 \normtwo{P_J(rg)}^2+\normtwo{(I-P_J)r}^2.
\end{equation}
Because $a\in V_J$, we have
$(I-P_J)r=(I-P_J)f$, so for a Besov-smooth regression function the second
term is a controlled approximation tail.  A decoupled empirical cross-product
can therefore provide a high-probability upper bound on the first term and
hence a valid confidence radius for the unweighted error.  This
finite-dimensional coercivity is the step that removes the need to estimate
the design density.

\subsection{Related literature}

The distinction between adaptive estimation and adaptive uncertainty
quantification is classical.  Honest confidence regions for nonparametric
regression were studied by Li \cite{Li1989}, while Low \cite{Low1997} and Cai
and Low \cite{CaiLow2004} developed general impossibility and adaptation
principles.  Work specifically on global $L^2$ confidence sets includes
Juditsky and Lambert-Lacroix \cite{JuditskyLambertLacroix2003}, Baraud
\cite{Baraud2004}, Genovese and Wasserman \cite{GenoveseWasserman2005}, Cai
and Low \cite{CaiLow2006}, and Robins and van der Vaart
\cite{RobinsVdV2006}.  The random-design regression application of Robins and
van der Vaart takes the design distribution as known and formulates the loss in
$L^2(P_X)$.  Here the design density is unknown, while honesty and diameter are
measured in the fixed, unweighted space $L^2(dx)$.

In density estimation, Bull and Nickl \cite{BullNickl2013} established the two
adaptation regimes that motivate the formulation used here: full adaptation
over a smoothness range no wider than a factor of two, and adaptation on
suitably separated parameter spaces over wider ranges.  The corresponding
separation scale is an instance of a minimax goodness-of-fit testing rate.
Related testing theory for Gaussian random-design regression was developed by
Ingster and Sapatinas \cite{IngsterSapatinas2009}.  
Carpentier \cite{Carpentier2015} studied the
closely related problem of testing a rough smoothness class against a nested
smoother class after deleting a critical neighborhood of the smoother class,
and Carpentier \cite{Carpentier2013} obtained analogous
adaptive-confidence-set results for $L^p$ loss in Gaussian regression.  A
different line of work replaces separation from smoother classes by
self-similarity, tail-regularity, or related restrictions that exclude statistically exceptional functions; see
\cite{GineNickl2010,Bull2012,HoffmannNickl2011,NicklSzabo2016,SzaboVdVVZ2015}.

The closest predecessor of the present paper is Mukherjee and Sen
\cite{MukherjeeSen2018}.  Their construction uses second-order wavelet
$U$-statistics together with an adaptive estimator of $g$, and their
confidence-set theorem imposes the design-smoothness condition described
above.  They discuss higher-order influence functions as one possible route to
weaker nuisance assumptions; such methods were developed for nonlinear
functionals and low-regularity nuisance problems by Robins and collaborators
\cite{RobinsEtAl2008,RobinsEtAl2017}.  We retain their binary-regression model,
unweighted loss, Besov regression classes, and separated-space formulation.
The new ingredient is the coercivity of $P_JM_gP_J$, which replaces the
design-density plug-in step and yields the same adaptation regimes uniformly
over all densities satisfying \eqref{eq:intro-overlap}.  The lower-bound
portion follows from a strategy similar to their sharpness
argument.  

Adaptive estimation in random-design regression is available under
much weaker conditions than those required by the earlier confidence-set
construction; see, for example, Antoniadis and Leblanc \cite{AntoniadisLeblanc2000} and Baraud \cite{Baraud2002}.

\subsection{Model, function spaces, and main result}

Fix an integer $d\ge1$ and set $\X=[0,1]^d$.  For each $n$, we observe independent pairs
$(X_i,Y_i)_{i=1}^n$ satisfying
\begin{equation}\label{eq:model}
 X_i\sim g,
 \qquad
 Y_i\mid X_i=x\sim\operatorname{Bernoulli}(f(x)).
\end{equation}
The design density belongs to
\begin{equation}\label{eq:G}
 \G(c,C)=
 \left\{g:\int_\X g(x)\,dx=1,\quad c\leq g\leq C
 \text{ almost everywhere}\right\},
 \qquad 0<c\leq1\leq C<\infty.
\end{equation}
The constants $M,c,C,\beta_-,\beta_+$ are treated as known throughout.
Adaptation refers to the unknown smoothness label $\beta$; we do not consider
adaptation to the Besov radius or to the overlap constants.

Fix $S>\beta_+$ and a compactly supported, tensor-product,
boundary-corrected orthonormal wavelet basis of Cohen--Daubechies--Vial type
\cite{CohenDaubechiesVial1993}, constructed from a one-dimensional wavelet
system with H\"older--Zygmund regularity $R_{\mathrm w}>S$ and an integer
number $N_{\mathrm w}>S$ of vanishing moments.  The properties of this basis
used below are recorded in \cref{app:wavelets}.  We fix a coarse resolution
level $J_0$ sufficiently large so that the fixed-prototype boundary description
in \cref{app:wavelets} holds at every level $j\geq J_0$.  This choice depends
only on the fixed one-dimensional wavelet system and its boundary correction.
For $J\geq J_0$, let $V_J$, $P_J$, and $Q_j=P_{j+1}-P_j$ denote the
approximation spaces and orthogonal projections specified there.  The
appendix also fixes the orthonormal scaling basis
$(\varphi_{J,\lambda})_{\lambda\in\Lambda_J}$ of $V_J$ and the dimension
$D_J=\dim(V_J)$.

For $s\geq0$, define
\begin{equation}\label{eq:Besov-norm}
 \norm{h}_{\mathbb B^s_{2,\infty}}
 =2^{J_0s}\normtwo{P_{J_0}h}
 +\sup_{j\geq J_0}2^{js}\normtwo{Q_jh}.
\end{equation}
Let $\mathbb B^s_{2,\infty}([0,1]^d)$ denote the space on which the norm in
\eqref{eq:Besov-norm} is finite.  Set
\begin{equation}\label{eq:function-classes}
 \B_s(M)=\left\{h\in L^2(dx):
 \norm{h}_{\mathbb B^s_{2,\infty}}\leq M\right\},
 \qquad
 \F_s(M)=\left\{f\in\B_s(M):0\leq f\leq1\text{ a.e.}\right\}.
\end{equation}
By \cref{prop:app-wavelet-estimates}, we have the exact nesting property
\begin{equation}\label{eq:nesting}
 t\geq s
 \quad\Longrightarrow\quad
 \B_t(M)\subseteq\B_s(M),
 \qquad
 \F_t(M)\subseteq\F_s(M).
\end{equation}
The same proposition gives a constant $A_B$, depending only on $\beta_-$ and
the chosen basis, such that, uniformly for
$s\in[\beta_-,\beta_+]$ and $J\geq J_0$,
\begin{equation}\label{eq:approximation}
 \normtwo{(I-P_J)h}\leq A_BM2^{-Js},
 \qquad h\in\B_s(M).
\end{equation}

For a nonempty set $C\subset L^2(dx)$, write
\begin{equation}\label{eq:diam-radius}
 \diam(C)=\sup_{u,v\in C}\normtwo{u-v},
\end{equation}
with $\diam(\varnothing) = 0$. 

We call a sequence of random sets $C_n$ \emph{honest} over a sequence of
parameter spaces $\Theta_n$ at level $1-\alpha$ if
\begin{equation}\label{eq:honesty-definition}
 \liminf_{n\to\infty}\inf_{(f,g)\in\Theta_n}
 \Pp_{f,g}(f\in C_n)\geq1-\alpha.
\end{equation}
It is \emph{adaptive} over a smoothness range if, for each $\beta$ in that
range, its diameter is of order $\eps_n(\beta)$ uniformly over the
corresponding class.

The estimation and testing radii are
\begin{equation}\label{eq:rates}
 \eps_n(s)=n^{-s/(2s+d)},
 \qquad
 \rho_n(s)=n^{-2s/(4s+d)}.
\end{equation}
As noted above,
\begin{equation}\label{eq:rate-identity}
 \eps_n(2s)^2=\rho_n(s)^2.
\end{equation}

When $\beta_+>2\beta_-$, define
\begin{equation}\label{eq:grid}
 N=\left\lceil\log_2(\beta_+/\beta_-)\right\rceil,
 \qquad
 s_j=2^{j-1}\beta_-,
 \quad j=1,\ldots,N.
\end{equation}
Then $s_N<\beta_+\leq2s_N$.  For $L>0$, define
\begin{equation}\label{eq:separated-space}
 \F_n^{\mathrm{sep}}(L)
 =
 \F_{s_N}(M)
 \cup
 \bigcup_{j=1}^{N-1}
 \left\{
 f\in\F_{s_j}(M):
 \dist\bigl(f,\F_{s_{j+1}}(M)\bigr)
 \geq L\rho_n(s_j)
 \right\},
\end{equation}
where
\begin{equation}\label{eq:distance-class}
 \dist(f,\mathcal A)
 =
 \inf_{h\in\mathcal A}\normtwo{f-h}.
\end{equation}
Thus each rougher class is separated, at its testing rate, from the next
smoother regression class.

We can now state the main theorem.

\begin{theorem}\label{thm:main}
Fix $0<\beta_-\leq\beta_+<S$, $M>0$,
$0<c\leq1\leq C<\infty$, and $0<\alpha,\alpha'<1$.

\begin{enumerate}[label=\textup{(\roman*)}]
\item If $\beta_+\leq2\beta_-$, there are random sets $C_n$ and a constant
$K<\infty$ such that
\begin{equation}\label{eq:main-narrow-coverage}
 \liminf_{n\to\infty}
 \inf_{\substack{\beta\in[\beta_-,\beta_+]\\
 f\in\F_\beta(M),\ g\in\G(c,C)}}
 \Pp_{f,g}(f\in C_n)\geq1-\alpha,
\end{equation}
and, for every $\beta\in[\beta_-,\beta_+]$,
\begin{equation}\label{eq:main-narrow-diameter}
 \limsup_{n\to\infty}
 \sup_{\substack{f\in\F_\beta(M)\\g\in\G(c,C)}}
 \Pp_{f,g}\left\{\diam(C_n)>K\eps_n(\beta)\right\}
 \leq\alpha'.
\end{equation}

\item If $\beta_+>2\beta_-$, there exist random sets $C_n$ and constants
$L_0,K<\infty$ such that, for every $L\geq L_0$,
\begin{equation}\label{eq:main-wide-coverage}
 \liminf_{n\to\infty}
 \inf_{\substack{f\in\F_n^{\mathrm{sep}}(L)\\g\in\G(c,C)}}
 \Pp_{f,g}(f\in C_n)\geq1-\alpha,
\end{equation}
and, for every $\beta\in[\beta_-,\beta_+]$,
\begin{equation}\label{eq:main-wide-diameter}
 \limsup_{n\to\infty}
 \sup_{\substack{f\in\F_n^{\mathrm{sep}}(L)\cap\F_\beta(M)\\
 g\in\G(c,C)}}
 \Pp_{f,g}\left\{\diam(C_n)>K\eps_n(\beta)\right\}
 \leq\alpha'.
\end{equation}
\end{enumerate}

All constants may depend only on the displayed fixed
parameters, $d$, and the chosen wavelet basis.
\end{theorem}

\begin{remark}\label{rem:comparison-MS}
The factor-of-two distinction, the separated-space formulation, and the order
of the testing radius in \cref{thm:main} come from the existing
adaptive-confidence-set theory; see
\cite{BullNickl2013,MukherjeeSen2018}.  The principal content specific to the
present paper is that the same conclusions hold uniformly over the full overlap
class $\G(c,C)$, without any positive smoothness assumption on the design
density.  The separation order is minimax-optimal for the associated testing problem.
For confidence sets, the corresponding sharpness result holds when
$2\alpha+\alpha'<1$; see \cref{thm:separation-sharp}.
\end{remark}

\begin{remark}\label{rem:bounded-responses}
The upper-bound construction uses the conditional distribution of the
response only through
\[
 0\leq Y\leq1\quad\text{almost surely},
 \qquad
 \Ep(Y\mid X)=f(X).
\]
Indeed, these conditions give the squared-loss identity used for validation,
the bound $\operatorname{Var}(Y\mid X)\leq1/4$ used in the pilot estimator risk 
argument, and the uniform boundedness required for the residual
cross-product estimate.  Thus the same upper-bound construction and proof
apply to random-design regression with any response supported on $[0,1]$
and conditional mean $f$.  We formulate the results for binary regression, where the conditional response law is determined entirely by $f$, to fix a concrete example with a matching lower bound.
\end{remark}

The following consequence, proved in \cref{sec:sharpness}, makes the
rough-design conclusion explicit.

\begin{corollary}\label{cor:threshold}
The conclusions of \cref{thm:main} hold uniformly over the full design class
$\G(c,C)$, without any positive Besov smoothness assumption on the design
density.  Moreover, when $c<1<C$, $\G(c,C)$ contains a density $g$ such
that
\[
 g\notin B^\gamma_{2,\infty}([0,1]^d)
 \qquad\text{for every }\gamma>0.
\]
\end{corollary}

\subsection{Proof ideas and organization}\label{sec:proof-ideas}

We first split the sample into three independent blocks, denoted by
$I_0,I_1$, and $I_2$, and then divide $I_2$ into two blocks $A$ and $B$.
On $I_0$ we construct a finite family of bounded wavelet least-squares
estimators, and on $I_1$ we select among them by independent validation.  The
resulting bounded adaptive pilot estimator $\widetilde f$, which depends only on
$I_0$ and $I_1$, satisfies
\begin{equation}\label{eq:proof-idea-pilot}
 \sup_{f\in\F_\beta(M),\,g\in\G(c,C)}
 \Ep_{f,g}\normtwo{\widetilde f-f}^2
 \lesssim n^{-2\beta/(2\beta+d)}
\end{equation}
uniformly over $\beta\in[\beta_-,\beta_+]$.

Given any bounded center function $a$ measurable with respect to $I_0$ and
$I_1$, we use $A$ and $B$ to form two independent empirical residual-moment
vectors.  Their inner product is a decoupled cross-product whose conditional
mean is
\[
 q_J=\normtwo{P_J((f-a)g)}^2
\]
and whose conditional variance is bounded above by
\[
 \frac{q_J}{\ell}+\frac{D_J}{\ell^2},
 \qquad
 \ell=|A|=|B|\asymp n.
\]
Adding a suitable fluctuation term to this statistic yields a
high-probability upper bound on $q_J$.  For the confidence ball we take
$a=P_J\widetilde f$, so the high-frequency part of $f-a$ is exactly the
Besov tail of $f$.  The coercivity inequality \eqref{eq:intro-bound} then
converts the observable bound into a confidence radius for $\normtwo{f-a}$
without estimating $g$.

At a base smoothness $s$, choosing
$2^J\asymp n^{2/(4s+d)}$ balances the squared Besov approximation tail with the
cross-product fluctuation; both are of order $\rho_n(s)^2$.  Together with
\eqref{eq:proof-idea-pilot} and
$\rho_n(s)=\eps_n(2s)$, this gives the optimal diameter throughout
$\beta\in[s,2s]$.  Over wider ranges, we project the pilot estimator onto the next
smoother ball and use the same residual bound to test adjacent smoothness
classes.  A finite sequence of such tests selects the appropriate dyadic shell,
after which the corresponding factor-of-two confidence ball is reported.

The sharpness proof uses a Rademacher wavelet mixture under the uniform design.
It is close in structure to the lower-bound argument of Mukherjee and Sen and
is included to verify the exact separation order for the classes used here.
The same section constructs a density in $\G(c,C)$ with no positive Besov
smoothness and records the elementary identifiability obstruction that arises
if the design vanishes on an open set.

\Cref{sec:main-proof} constructs the confidence sets and proves
\cref{thm:main}.  The pilot and residual-bound ingredients are proved in
\cref{sec:pilot-proof,sec:bound-proof}, respectively, and
\cref{sec:sharpness} contains the lower bound and the design-density results.

\subsection{Conventions}

Throughout, constants denoted by $c_0,c_1,\ldots$ or $C_0,C_1,\ldots$ are
positive and finite; their values may change from line to line.  Unless stated
otherwise, they depend only on the fixed model parameters, $d$, and the chosen
wavelet basis, never on $n$, $J$, $f$, or $g$.  We write $a_n\lesssim b_n$ if
$a_n\leq Cb_n$ for such a constant $C$, and $a_n\asymp b_n$ if both
$a_n\lesssim b_n$ and $b_n\lesssim a_n$.

Whenever pointwise evaluation of an element $h\in L^2(dx)$ is required, we
use the following fixed representative. For $x\in[0,1]^d$, take the limsup
of the averages of $h$ over the nested half-open dyadic cubes containing
$x$, assigning the value zero if this limsup is not finite. For members of
$\F_s(M)$, clip the resulting value to $[0,1]$. The resulting map is jointly
measurable in $(h,x)$ and agrees with $h$ almost everywhere. Since every $g\in\G(c,C)$ is absolutely continuous with respect
to Lebesgue measure, this convention does not alter the statistical law.
Confidence sets are understood as Effros-measurable random closed subsets of
$L^2(dx)$, meaning that for every open set $O$, the event
$\{C_n\cap O\neq\varnothing\}$ is measurable.  This entails no loss in the
lower bounds, since taking closures preserves diameter and can only improve coverage.

\subsection{Acknowledgments}
P.L. was partially supported by NSF grant DMS-2450004. This paper was written by the authors with the assistance of large language models, which included suggesting arguments, contributing to drafting and revision, and performing exploratory computational checks.

\section{Construction and proof of the main theorem}\label{sec:main-proof}

The construction rests on two ingredients: a rate-adaptive pilot estimator and a
residual cross-product bound.  We state both ingredients, construct the
confidence sets, and prove \cref{thm:main}; their corresponding proofs are postponed to
\cref{sec:pilot-proof,sec:bound-proof}.

We use the vector $\Phi_J$ of level-$J$ scaling functions, the associated projection kernel $K_J$, and the dimension $D_J= 2^{Jd}$
defined in  \eqref{eq:app-dimension} and \eqref{eq:app-Phi-kernel}.  The kernel
bounds and $L^\infty$ stability needed below are
\eqref{eq:app-kernel-bounds} and \eqref{eq:app-Linfty-stability}.

Split the observations into three disjoint blocks $I_0,I_1,I_2$ of size
$m=\lfloor n/3\rfloor$, and split $I_2$ into two blocks $A,B$ of common size
$\ell=\lfloor m/2\rfloor$.  Any remaining observations are discarded.  The
blocks $I_0,I_1$ construct the pilot; all residual cross-products use $A,B$.
For the finitely many $n$ for which the constructions below are unavailable,
we set $C_n=L^2(dx)$, which has no effect on the asymptotic
statements.

\subsection{Pilot construction and risk bound}

For each resolution $J$, we fit a least-squares wavelet series estimator on
$I_0$, set ill-conditioned fits to the zero function, clip the result to the Bernoulli
parameter space, and select among the surviving candidates by independent
validation on $I_1$.

Let $J_{\max}$ be the largest integer $J\geq J_0$ such that
\begin{equation}\label{eq:Jmax}
 D_J\leq m^{d/(2\beta_-+d)},
\end{equation}
and let $\mathcal J_m=\{J_0,\ldots,J_{\max}\}$.  For every $J\in\mathcal J_m$, define
on $I_0$
\begin{equation}\label{eq:empirical-gram}
 \widehat G_J=\frac1m\sum_{i\in I_0}
 \Phi_J(X_i)\Phi_J(X_i)^\top,
 \qquad
 \widehat b_J=\frac1m\sum_{i\in I_0}Y_i\Phi_J(X_i),
\end{equation}
and set
\begin{equation}\label{eq:pilot-candidates}
 \widehat\theta_J=
 \begin{cases}
  \widehat G_J^{-1}\widehat b_J,
  &\lambda_{\min}(\widehat G_J)\geq c/2,\\
  0,&\lambda_{\min}(\widehat G_J)<c/2,
 \end{cases}
 \qquad
 \widehat f_J=\clip\bigl(\Phi_J^\top\widehat\theta_J\bigr).
\end{equation}
On the independent block $I_1$, define
\begin{equation}\label{eq:validation-loss}
 \widehat L(J)=\frac1m\sum_{i\in I_1}
 \{Y_i-\widehat f_J(X_i)\}^2.
\end{equation}
Let $\widehat J$ be the smallest minimizer of $\widehat L(J)$ over
$\mathcal J_m$, and put
\begin{equation}\label{eq:pilot-definition}
 \widetilde f=\widehat f_{\widehat J}.
\end{equation}
The proof of the following proposition is given in \Cref{sec:pilot-proof}. 
\begin{proposition}\label{prop:pilot}
The estimator $\widetilde f$ is measurable, takes values in $[0,1]$, and
satisfies, uniformly for all $\beta\in[\beta_-,\beta_+]$, 
\begin{equation}\label{eq:pilot-risk}
 \sup_{\substack{f\in\F_\beta(M)\\g\in\G(c,C)}}
 \Ep_{f,g}\normtwo{\widetilde f-f}^2
 \lesssim n^{-2\beta/(2\beta+d)}.
\end{equation}
\end{proposition}

\subsection{The residual bound: statements}

Let $a$ be a bounded function measurable with respect to the pilot blocks, $I_0$ and $I_1$, and
put
\begin{equation}\label{eq:q-definition}
 r=f-a,
 \qquad
 q_J(r;g)=\normtwo{P_J(rg)}^2.
\end{equation}
The following deterministic lemma converts the
 quantity $q_J(r;g)$ targeted by the cross-product statistic into control of the unweighted error
$\normtwo r^2$. We prove it in \Cref{sec:bound-proof}.

\begin{lemma}\label{lem:coercivity}
For every $r\in L^2(dx)$ and $g\in\G(c,C)$,
\begin{align}
 q_J(r;g)^{1/2}
 &\leq C\normtwo r,
 \label{eq:q-upper-deterministic}\\
 q_J(r;g)^{1/2}
 &\geq c\normtwo{P_Jr}-C\normtwo{(I-P_J)r},
 \label{eq:coercive-one}\\
 q_J(r;g)^{1/2}
 &\geq c\normtwo r-(c+C)\normtwo{(I-P_J)r},
 \label{eq:coercive-total}\\
 \normtwo r^2
 &\leq\frac{2}{c^2}q_J(r;g)
 +\left(1+\frac{2C^2}{c^2}\right)
 \normtwo{(I-P_J)r}^2.
 \label{eq:coercive-square}
\end{align}
\end{lemma}

For $H\in\{A,B\}$, define
\begin{equation}\label{eq:residual-average}
 S_{H,J}(a)=\frac1\ell\sum_{i\in H}
 \{Y_i-a(X_i)\}\Phi_J(X_i)
\end{equation}
and the decoupled cross-product
\begin{equation}\label{eq:cross-product}
 T_J(a)=\ip{S_{A,J}(a)}{S_{B,J}(a)}.
\end{equation}
The next lemma is also proved in \Cref{sec:bound-proof}. 
\begin{lemma}\label{lem:cross-product-bound}
Fix a regression function $f$ with $0\leq f\leq1$, a design density
$g\in\G(c,C)$, and a resolution level $J\geq J_0$.  Let $\mathcal H$ be a
$\sigma$-field independent of the observations indexed by $A\cup B$, and let
$a$ be an $\mathcal H$-measurable function satisfying
$\norm a_\infty\leq A_0$ almost surely.  Set $r=f-a$.  Then, almost surely,
\begin{equation}\label{eq:T-moments}
 \Ep(T_J(a)\mid\mathcal H)=q_J(r;g),
 \qquad
 \operatorname{Var}(T_J(a)\mid\mathcal H)
 \leq V\left(\frac{q_J(r;g)}\ell+\frac{D_J}{\ell^2}\right),
\end{equation}
where $V$ depends only on $A_0$ and $C$.

Consequently, for every $0<\zeta<1$ there exists
$\lambda_\zeta<\infty$, depending only on $\zeta$, $A_0$, and $C$, such
that, with
\begin{equation}\label{eq:qbar}
 t_J=\lambda_\zeta\frac{\sqrt{D_J}}\ell,
 \qquad
 \overline q_J(a)=2(T_J(a)+t_J)_+,
\end{equation}
one has
\begin{equation}\label{eq:bound-event}
 \Pp\left(
 q_J(r;g)\leq\overline q_J(a)
 \leq3q_J(r;g)+4t_J
 \ \middle|\ \mathcal H
 \right)\geq1-\zeta
 \quad\text{almost surely}.
\end{equation}
The constants are uniform over all admissible choices of
$f$, $g$, $\mathcal H$, $a$, and $J$ satisfying the displayed conditions.
\end{lemma}

\subsection{A confidence ball for one factor-of-two smoothness range}

For $s\in[\beta_-,\beta_+]$, choose an integer $J_n(s)\geq J_0$ such that
\begin{equation}\label{eq:Jns}
 \left|J_n(s)-\frac{2}{4s+d}\log_2 n\right|\leq1.
\end{equation}
Since $\ell\asymp n$ and $D_J=2^{Jd}$ by
\eqref{eq:app-dimension}, uniformly for $s$ in
the fixed compact interval $[\beta_-,\beta_+]$,
\begin{equation}\label{eq:balancing}
 2^{-2J_n(s)s}
 \asymp
 \frac{\sqrt{D_{J_n(s)}}}{\ell}
 \asymp
 \rho_n(s)^2.
\end{equation}
Fix $0<\zeta<1$.  In every use of \eqref{eq:qbar}, take the constant
$\lambda_\zeta$ from \cref{lem:cross-product-bound} with
$A_0=\max\{A_\psi,1\}$, where $A_\psi$ is given by \cref{prop:app-wavelet-estimates}.

For a fixed $s$, abbreviate $J=J_n(s)$ and define
\begin{equation}\label{eq:block-center}
 a_s=P_J\widetilde f.
\end{equation}
By \eqref{eq:app-Linfty-stability},
$\norm{a_s}_\infty\leq A_\psi$.  Put
\begin{align}
 R_{n,s}^2
 &=\frac{2}{c^2}\overline q_J(a_s)
 +\left(1+\frac{2C^2}{c^2}\right)
 A_B^2M^22^{-2Js},
 \label{eq:block-radius}\\
 C_{n,s}
 &=\left\{h\in L^2(dx):\normtwo{h-a_s}\leq R_{n,s}\right\}.
 \label{eq:block-ball}
\end{align}

\begin{proposition}\label{prop:one-block}
For every fixed $s\in[\beta_-,\beta_+]$,
\begin{equation}\label{eq:one-block-coverage}
 \inf_{\substack{f\in\F_s(M)\\g\in\G(c,C)}}
 \Pp_{f,g}(f\in C_{n,s})\geq1-\zeta
\end{equation}
for all sufficiently large $n$.  Moreover, for every $\eta>0$, there is a constant
$K=K(\eta,s)<\infty$ such that
\begin{equation}\label{eq:one-block-diameter}
 \limsup_{n\to\infty}
 \sup_{\substack{\beta\in[s,\min\{2s,\beta_+\}]\\
 f\in\F_\beta(M),\ g\in\G(c,C)}}
 \Pp_{f,g}\left\{\diam(C_{n,s})>K\eps_n(\beta)\right\}
 \leq\zeta+\eta.
\end{equation}
For any fixed finite collection of values of $s$, the constant $K$ may be
chosen uniformly by taking the maximum of the corresponding constants.
\end{proposition}

\begin{proof}[Proof of \Cref{prop:one-block}]
Let $r=f-a_s$.  Since $a_s\in V_J$,
\begin{equation}\label{eq:block-tail-identity}
 (I-P_J)r=(I-P_J)f.
\end{equation}
If $f\in\F_s(M)$, then \eqref{eq:approximation} gives
\begin{equation}\label{e:newtail}
 \normtwo{(I-P_J)r}^2
 \leq A_B^2M^22^{-2Js}.
\end{equation}
Conditional on $a_s$, \eqref{eq:bound-event}, applied with
$a=a_s$ and $r=f-a_s$, has probability at least $1-\zeta$.  On this event,
\[
 q_J(r;g)\leq\overline q_J(a_s).
\]
Applying \eqref{eq:coercive-square} and the tail bound \eqref{e:newtail}, we obtain
\begin{align*}
 \normtwo{f-a_s}^2
 &=\normtwo r^2\\
 &\leq
 \frac{2}{c^2}q_J(r;g)
 +\left(1+\frac{2C^2}{c^2}\right)
 \normtwo{(I-P_J)r}^2\\
 &\leq
 \frac{2}{c^2}\overline q_J(a_s)
 +\left(1+\frac{2C^2}{c^2}\right)
 A_B^2M^22^{-2Js}\\
 &=R_{n,s}^2.
\end{align*}
Hence $f\in C_{n,s}$ on the event in \eqref{eq:bound-event}.  Taking
expectations of its conditional probability given $a_s$ yields
\[
 \Pp_{f,g}(f\in C_{n,s})\geq1-\zeta,
\]
uniformly over $f\in\F_s(M)$ and $g\in\G(c,C)$.  This proves
\eqref{eq:one-block-coverage}.

Now suppose that $f\in\F_\beta(M)$ with
$\beta\in[s,\min\{2s,\beta_+\}]$, and let
\[
 \mathcal E_{n,s}
 =
 \left\{
 q_J(f-a_s;g)
 \leq\overline q_J(a_s)
 \leq3q_J(f-a_s;g)+4t_J
 \right\}.
\]

Conditional on $a_s$, \eqref{eq:bound-event} gives
$\Pp(\mathcal E_{n,s}\mid a_s)\geq1-\zeta$.  On $\mathcal E_{n,s}$, the
upper inequality in its definition, \eqref{eq:q-upper-deterministic}, and
\eqref{eq:balancing} give
\begin{align}
 R_{n,s}^2
 &\lesssim
 q_J(f-a_s;g)+t_J+2^{-2Js}
 \notag\\
 &\lesssim
 \normtwo{f-a_s}^2+\rho_n(s)^2
 =
 \normtwo{f-P_J\widetilde f}^2+\rho_n(s)^2.
 \label{eq:block-radius-bound-1}
\end{align}
Moreover,
\[
 f-P_J\widetilde f
 =
 (I-P_J)f+P_J(f-\widetilde f),
\]
and the two terms on the right are orthogonal.  Hence
\begin{align}
 \normtwo{f-P_J\widetilde f}^2
 &=\normtwo{(I-P_J)f}^2
   +\normtwo{P_J(f-\widetilde f)}^2
 \notag\\
 &\leq A_B^2M^22^{-2J\beta}+\normtwo{f-\widetilde f}^2.
 \label{eq:block-orthogonal-bound}
\end{align}
Since $\beta\geq s$,
\[
 2^{-2J\beta}\leq2^{-2Js}\lesssim\rho_n(s)^2,
\]
uniformly over $\beta\in[s,\min\{2s,\beta_+\}]$.  Also,
\begin{equation}\label{eq:rate-comparison}
 \rho_n(s)^2\leq\eps_n(\beta)^2
 \quad\Longleftrightarrow\quad
 \beta\leq2s.
\end{equation}
Consequently, on $\mathcal E_{n,s}$,
\[
 R_{n,s}^2
 \leq C_0\left(
 \normtwo{f-\widetilde f}^2+\eps_n(\beta)^2
 \right),
\]
where $C_0$ is independent of $\beta$ in the displayed range.  Since
$\diam(C_{n,s})=2R_{n,s}$, the uniform pilot bound in
\cref{prop:pilot} and Markov's inequality imply that one may choose
$K=K(\eta,s)$, independent of $\beta$, such that
\[
 \limsup_{n\to\infty}
 \sup_{\substack{\beta\in[s,\min\{2s,\beta_+\}]\\
 f\in\F_\beta(M),\ g\in\G(c,C)}}
 \Pp_{f,g}\left\{
 \diam(C_{n,s})>K\eps_n(\beta),\ \mathcal E_{n,s}
 \right\}
 \leq\eta.
\]
Finally, averaging
$\Pp_{f,g}(\mathcal E_{n,s}^c\mid a_s)\leq\zeta$ gives
$\Pp_{f,g}(\mathcal E_{n,s}^c)\leq\zeta$ uniformly over the same parameter
range.  The union bound therefore proves
\eqref{eq:one-block-diameter}.
\end{proof}

\subsection{Tests between adjacent smoothness classes}

For $s$ such that $2s\leq\beta_+$, the set $\F_{2s}(M)$ is nonempty, closed,
and convex in $L^2(dx)$.  Indeed, it contains zero; convexity is immediate;
the coefficient inequalities defining the Besov ball pass to $L^2$ limits,
as do the constraints $0\leq f\leq1$ after taking a subsequence converging
almost everywhere.  Let
\begin{equation}\label{eq:metric-projection}
 \widehat h_s=\operatorname*{argmin}_{h\in\F_{2s}(M)}
 \normtwo{\widetilde f-h}
\end{equation}
be the Hilbert-space metric projection.  It is unique and is a measurable
function of $\widetilde f$ because projection onto a closed convex subset of a
separable Hilbert space is continuous.  Its chosen representative takes
values in $[0,1]$.

Choose $n_{\mathrm{test}}<\infty$, depending only on the fixed smoothness
range, $d$, and $J_0$, so that $J_n(s_j)$ is defined for every $j=1,\ldots,N$ and
every $n\geq n_{\mathrm{test}}$. Let $\tau>0$ be a threshold parameter, to
be chosen below.  For $n\geq n_{\mathrm{test}}$ and $j<N$, define
\begin{equation}\label{eq:adjacent-test}
 \phi_j=
 \1\left\{
 \overline q_{J_n(s_j)}(\widehat h_{s_j})
 >\tau\rho_n(s_j)^2
 \right\}.
\end{equation}
Here $\phi_j=1$ means that the smoother null
$f\in\F_{2s_j}(M)$ is rejected.  For $n<n_{\mathrm{test}}$, we use the
standing convention that $C_n=L^2(dx)$, so the tests need not be defined.

\begin{lemma}\label{lem:adjacent-tests}
Fix $0<\zeta<1$.  There exist
$n_0\geq n_{\mathrm{test}}$ and constants
$A_1,B_\zeta,B_0<\infty$, depending only on $\zeta$ and the fixed model and
wavelet parameters, such that the following bounds hold for every
$n\geq n_0$ and every $j<N$.  

If $\tau>4B_\zeta$, then
\begin{equation}\label{eq:type-one}
 \sup_{\substack{f\in\F_{2s_j}(M)\\g\in\G(c,C)}}
 \Pp_{f,g}(\phi_j=1)
 \leq
 \zeta+\frac{A_1}{\tau-4B_\zeta}.
\end{equation}
If $cL>B_0$ and $(cL-B_0)^2>\tau$, then
\begin{equation}\label{eq:type-two}
 \sup_{\substack{f\in\F_{s_j}(M),\ 
 \dist(f,\F_{2s_j}(M))\geq L\rho_n(s_j)\\
 g\in\G(c,C)}}
 \Pp_{f,g}(\phi_j=0)
 \leq\zeta.
\end{equation}
\end{lemma}

\begin{proof}[Proof of \Cref{lem:adjacent-tests}]
Write $s=s_j$, $J=J_n(s)$, $\rho=\rho_n(s)$, and
\begin{equation}\label{eq:test-q}
 q=q_J(f-\widehat h_s;g).
\end{equation}
After increasing $n_0$ if necessary, \eqref{eq:balancing} gives
\[
 t_J\leq B_\zeta\rho^2
\]
for every $n\geq n_0$ and every $j<N$.  Let
\[
 \mathcal E_j
 =
 \left\{
 q\leq\overline q_J(\widehat h_s)
 \leq3q+4t_J
 \right\}.
\]
Conditional on $\widehat h_s$, \eqref{eq:bound-event}, applied with
$a=\widehat h_s$ and $r=f-\widehat h_s$, gives
\[
 \Pp_{f,g}(\mathcal E_j\mid\widehat h_s)\geq1-\zeta.
\]
Consequently,
\[
 \Pp_{f,g}(\mathcal E_j^c)\leq\zeta.
\]

Suppose first that $f\in\F_{2s}(M)$.  Since $f$ is feasible in
\eqref{eq:metric-projection},
\begin{equation}\label{eq:projection-under-null}
 \normtwo{\widetilde f-\widehat h_s}
 \leq\normtwo{\widetilde f-f},
 \qquad
 \normtwo{f-\widehat h_s}
 \leq2\normtwo{f-\widetilde f}.
\end{equation}
By \eqref{eq:q-upper-deterministic}, \cref{prop:pilot}, and
\eqref{eq:rate-identity},
\begin{equation}\label{eq:null-q-mean}
 \Ep_{f,g}q\lesssim\eps_n(2s)^2=\rho^2.
\end{equation}
On $\mathcal E_j$, the bound $t_J\leq B_\zeta\rho^2$ shows that the event
$\{\phi_j=1\}$ implies
\[
 q>\frac{\tau-4B_\zeta}{3}\rho^2.
\]
Thus \eqref{eq:null-q-mean} and Markov's inequality give
\[
 \Pp_{f,g}(\phi_j=1)
 \leq\zeta+\frac{A_1}{\tau-4B_\zeta}
\]
for a constant $A_1$ independent of $j$, proving \eqref{eq:type-one}.

Now suppose that $f\in\F_s(M)$ and
$\dist(f,\F_{2s}(M))\geq L\rho$.  Since
$\widehat h_s\in\F_{2s}(M)\subset\F_{s}(M)$,
\begin{equation}\label{eq:alternative-full-distance}
 \normtwo{f-\widehat h_s}\geq L\rho.
\end{equation}
Also $f,\widehat h_s\in\B_s(M)$ by the nesting property \eqref{eq:nesting}, and hence \eqref{eq:approximation} yields
\begin{equation}\label{eq:alternative-tail}
 \normtwo{(I-P_J)(f-\widehat h_s)}
 \leq2A_BM2^{-Js}\lesssim\rho.
\end{equation}
Applying \eqref{eq:coercive-total} to
$r=f-\widehat h_s$, and using
\eqref{eq:alternative-full-distance}--\eqref{eq:alternative-tail}, gives
\begin{align*}
 q^{1/2}
 &\geq
 c\normtwo{f-\widehat h_s}
 -(c+C)\normtwo{(I-P_J)(f-\widehat h_s)}\\
 &\geq
 cL\rho-2(c+C)A_BM2^{-Js}\\
 &\geq
 (cL-B_0)\rho,
\end{align*}
where $B_0<\infty$ is independent of $j$ by \eqref{eq:balancing}.  Under
the stated conditions on $L$, this implies
\[
 q\geq(cL-B_0)^2\rho^2>\tau\rho^2.
\]
On $\mathcal E_j$, we have
$\overline q_J(\widehat h_s)\geq q$, and hence $\phi_j=1$.  
Thus
\[
 \{\phi_j=0\}\subseteq\mathcal E_j^c.
\]
Since $\Pp_{f,g}(\mathcal E_j^c)\leq\zeta$, it follows that
$\Pp_{f,g}(\phi_j=0)\leq\zeta$, proving \eqref{eq:type-two}.
\end{proof}

\subsection{Assembly of the confidence set}

\begin{proof}[Proof of \Cref{thm:main}]
Suppose first that $\beta_+\leq2\beta_-$, and set $s=\beta_-$.  Choose
$\zeta<\min\{\alpha,\alpha'\}/2$ and define
\begin{equation}\label{eq:narrow-final-set}
 C_n=C_{n,s}.
\end{equation}
By \eqref{eq:nesting}, $\F_\beta(M)\subseteq\F_s(M)$ for every
$\beta\in[s,\beta_+]$, so \eqref{eq:one-block-coverage} gives honesty
uniformly over the full smoothness range.

For the diameter, the assumption $\beta_+\leq2\beta_-$ gives
\[
 [\beta_-,\beta_+]\subseteq[s,2s].
\]
Hence the uniform bound \eqref{eq:one-block-diameter} applies throughout the
full smoothness range.  Choose $\eta>0$ so that
$\zeta+\eta<\alpha'$.  Then, for the corresponding constant
$K=K(\eta,s)$,
\[
 \limsup_{n\to\infty}
 \sup_{\substack{\beta\in[\beta_-,\beta_+]\\
 f\in\F_\beta(M),\ g\in\G(c,C)}}
 \Pp_{f,g}\left\{\diam(C_n)>K\eps_n(\beta)\right\}
 \leq\zeta+\eta<\alpha'.
\]
This proves part~(i).

Now suppose that $\beta_+>2\beta_-$.  For $j<N$, let
\begin{equation}\label{eq:shells}
 \mathcal S_j(L)=
 \left\{f\in\F_{s_j}(M):
 \dist(f,\F_{s_{j+1}}(M))\geq L\rho_n(s_j)\right\},
 \qquad
 \mathcal S_N=\F_{s_N}(M).
\end{equation}
These shells are disjoint.  Indeed, membership in a later shell implies
membership in $\F_{s_{j+1}}(M)$ for every earlier index $j$, and hence rules
out the earlier positive-separation condition. 
Thus every $f\in\F_n^{\mathrm{sep}}(L)$ has a unique shell index, denoted
$k(f)$.

For $n<n_{\mathrm{test}}$, retain the standing convention
$C_n=L^2(dx)$.  We therefore assume $n\geq n_{\mathrm{test}}$ throughout
the remainder of the construction. 
The wide-range procedure has four steps:
\begin{enumerate}
\item  construct the single-range confidence ball $C_{n,s_j}$ for each
      dyadic smoothness range;
\item test against each smoother adjacent class $\F_{s_{j+1}}(M)$;
\item choose the first index at which the smoother class is rejected, using
      the top index if no test rejects;
\item report the confidence ball attached to that selected index.
\end{enumerate}
Accordingly, define
\begin{equation}\label{eq:selected-shell}
 \widehat k=\min\{j<N:\phi_j=1\},
\end{equation}
with $\widehat k=N$ if no test rejects, and set
\begin{equation}\label{eq:wide-final-set}
 C_n=C_{n,s_{\widehat k}}.
\end{equation}
If the truth lies in shell $k$, then tests $1,\ldots,k-1$ are under their
smoother nulls, while test $k$ is under its separated alternative when
$k<N$.  This is the ideal rejection pattern that the selection rule seeks to
recover.

We now choose the constants in the following order.  First choose $\zeta>0$
so small that
\begin{equation}\label{eq:zeta-choice}
 (N+1)\zeta\leq\alpha,
 \qquad
 (N+2)\zeta\leq\frac{\alpha'}3.
\end{equation}
This determines $\lambda_\zeta$ and $B_\zeta$.  Next choose
$\tau>4B_\zeta$ so large that
\begin{equation}\label{eq:tau-choice}
 N\frac{A_1}{\tau-4B_\zeta}\leq\frac{\alpha'}3.
\end{equation}
Finally choose $L_0$ so large that
\begin{equation}\label{eq:L0-choice}
 cL_0>B_0,
 \qquad
 (cL_0-B_0)^2>\tau.
\end{equation}
Fix any $L\geq L_0$.

The same residual blocks $A$ and $B$ are used to form
$\overline q_{J_n(s_j)}(a_{s_j})$ for the confidence balls and
$\overline q_{J_n(s_j)}(\widehat h_{s_j})$ for the adjacent tests, for all
$j$.  This does not require independence across indices: both $a_{s_j}$ and
$\widehat h_{s_j}$ are measurable with respect to the pilot blocks
$I_0,I_1$, while $A,B$ are independent of those blocks.  Conditional on
$I_0,I_1$, \cref{lem:cross-product-bound} therefore applies separately at
each index, and the simultaneous and selection-error bounds below follow by
union bounds over the finite grid.

For fixed $(f,g)$, define the simultaneous residual-bound event
\begin{equation}\label{eq:simultaneous-bound-event}
\begin{aligned}
 \mathcal E_n(f,g)
 =
 \bigcap_{j=1}^N
 \Bigl\{
 &q_{J_n(s_j)}(f-a_{s_j};g)
 \leq
 \overline q_{J_n(s_j)}(a_{s_j})
 \leq
 3q_{J_n(s_j)}(f-a_{s_j};g)
 +4t_{J_n(s_j)}
 \Bigr\}.
\end{aligned}
\end{equation}
Conditional on the pilot blocks, \cref{lem:cross-product-bound} and a union
bound give
\begin{equation}\label{eq:simultaneous-bound-probability}
 \Pp_{f,g}\bigl\{\mathcal E_n(f,g)\mid I_0,I_1\bigr\}
 \geq1-N\zeta.
\end{equation}
In particular,
\[
 \Pp_{f,g}\{\mathcal E_n(f,g)\}\geq1-N\zeta.
\]

We first prove coverage.  Let $k=k(f)$.  On
$\mathcal E_n(f,g)$, if $\widehat k\leq k$, then
\[
 f\in\F_{s_k}(M)\subseteq\F_{s_{\widehat k}}(M). 
\]
Since $a_{s_{\widehat k}}\in V_{J_n(s_{\widehat k})}$, we have
\[
 (I-P_{J_n(s_{\widehat k})})
 (f-a_{s_{\widehat k}})
 =
 (I-P_{J_n(s_{\widehat k})})f.
\]
The inequality defining $\mathcal E_n(f,g)$, together with
\eqref{eq:coercive-square}, \eqref{eq:approximation}, and
\eqref{eq:block-radius}, therefore gives
\[
 f\in C_{n,s_{\widehat k}}.
\]
If $k<N$, the only remaining way coverage can fail is if
$\widehat k>k$, which requires $\phi_k=0$; by \eqref{eq:type-two}, this event
has probability at most $\zeta$.  If $k=N$, the inequality
$\widehat k>k$ is impossible.  Consequently,
\begin{equation}\label{eq:wide-coverage-bound}
 \sup_{\substack{f\in\F_n^{\mathrm{sep}}(L)\\g\in\G(c,C)}}
 \Pp_{f,g}(f\notin C_n)
 \leq
 N\zeta+\zeta
 =
 (N+1)\zeta
\end{equation}
for all sufficiently large $n$.  By \eqref{eq:zeta-choice}, the right-hand
side is at most $\alpha$.

For the diameter, first bound the shell-selection error.  If the truth lies
in shell $k=k(f)$, then
\[
 \{\widehat k\neq k\}
 \subseteq
 \bigcup_{j<k}\{\phi_j=1\}
 \,\cup\,
 \bigl(\{k<N\}\cap\{\phi_k=0\}\bigr).
\]
Indeed, a rejection before index $k$ makes $\widehat k<k$, while, if no
earlier test rejects, the inequality $\widehat k>k$ can occur only when
$k<N$ and test $k$ fails to reject.  Applying
\cref{lem:adjacent-tests} and a union bound therefore gives
\begin{equation}\label{eq:selection-error}
 \sup_{\substack{f\in\F_n^{\mathrm{sep}}(L)\\g\in\G(c,C)}}
 \Pp_{f,g}\{\widehat k\neq k(f)\}
 \leq
 (N+1)\zeta+
 N\frac{A_1}{\tau-4B_\zeta}.
\end{equation}
By \eqref{eq:zeta-choice} and \eqref{eq:tau-choice}, this is at most $2\alpha'/3$.

On the event $\{\widehat k=k(f)\}$, the reported set is
$C_n=C_{n,s_{k(f)}}$.  Consequently, for every $K>0$,
\[
 \left\{\diam(C_n)>K\eps_n(\beta)\right\}
 \subseteq
 \{\widehat k\neq k(f)\}
 \cup
 \left\{
 \diam(C_{n,s_{k(f)}})>K\eps_n(\beta)
 \right\}.
\]
It remains to verify that the second event is controlled by
\cref{prop:one-block} for every possible relation between $\beta$ and the
true shell smoothness. 
Write $s=s_{k(f)}$ and
suppose $f\in\F_\beta(M)$.  
If $\beta<s$, then shell membership gives $f\in\F_s(M)$.  Since
$\eps_n(s)\leq\eps_n(\beta)$,
\[
 \left\{
 \diam(C_{n,s})>K\eps_n(\beta)
 \right\}
 \subseteq
 \left\{
 \diam(C_{n,s})>K\eps_n(s)
 \right\},
\]
so \cref{prop:one-block}, applied at the smoothness label $s$, gives the
required bound.  If $\beta\geq s$ and $k(f)<N$, then $\beta<2s$, since
otherwise nesting would imply
$f\in\F_{2s}(M)$, contradicting the defining separation of the shell.  At
the top shell, $\beta\leq\beta_+\leq2s_N$.  Thus
\cref{prop:one-block} applies in every case. 
Take the maximum of its constants over the finite
grid and apply the one-block bound with auxiliary error $\eta=\alpha'/3$.
Combining it with \eqref{eq:selection-error}, the limiting diameter error is at most
\begin{equation*}
 (N+2)\zeta
 +N\frac{A_1}{\tau-4B_\zeta}
 +\frac{\alpha'}3
 \leq\alpha'
\end{equation*}
by \eqref{eq:zeta-choice} and \eqref{eq:tau-choice}.  This proves
\eqref{eq:main-wide-diameter} and completes part~(ii).
\end{proof}

\section{Proof of the adaptive pilot bound}\label{sec:pilot-proof}

This section proves \cref{prop:pilot}.  The wavelet properties used below are
collected in \cref{prop:app-wavelet-estimates}.

\subsection{A finite-list validation bound}

Conditional on any bounded candidate $h$ independent of the validation
sample,
\[
 \Ep\left(\{Y-h(X)\}^2\mid h\right)
 =
 \Ep\{f(X)(1-f(X))\}+\normg{h-f}^2.
\]
We see that comparing validation losses is equivalent, conditionally on the
candidates, to comparing their $L^2(g)$ risks.

\begin{lemma}\label{lem:validation}
For a positive integer $K_0$, let $h_1,\ldots,h_{K_0}$ be random measurable
functions taking values in
$[0,1]$, independent of an i.i.d.\ validation sample
$(X_i,Y_i)_{i=1}^m$ from \eqref{eq:model}.  Let $\widehat k$ be the smallest
minimizer of
\begin{equation*}
 \frac1m\sum_{i=1}^m\{Y_i-h_k(X_i)\}^2
\end{equation*}
over $1\leq k\leq K_0$.  Then, for a universal constant $C_0$,
\begin{equation}\label{eq:validation-oracle}
 \Ep\normg{h_{\widehat k}-f}^2
 \leq
 3\min_{1\leq k\leq K_0}\Ep\normg{h_k-f}^2
 +C_0\frac{\log(2K_0m)}m.
\end{equation}
\end{lemma}

\begin{proof}
Let $\mathcal H=\sigma(h_1,\ldots,h_{K_0})$ and condition on
$\mathcal H$.  For each $k$, define
\begin{equation*}
 \Delta_k(X,Y)=\{Y-h_k(X)\}^2-\{Y-f(X)\}^2,
 \qquad
 R_k=\Ep(\Delta_k\mid\mathcal H)=\normg{h_k-f}^2.
\end{equation*}
The subtraction of the common empirical loss of $f$ does not change the
minimizer.  Because $Y,h_k,f\in[0,1]$,
\begin{equation}\label{eq:Bernstein-properties}
 |\Delta_k|\leq1,
 \qquad
 \Ep(\Delta_k^2\mid\mathcal H)\leq4R_k.
\end{equation}
Bernstein's inequality, followed by
$\sqrt{8R_kx/m}\leq R_k/2+4x/m$, shows that for a universal $C_1$,
\begin{equation}\label{eq:validation-concentration}
 \Pp\left(
 \left|\frac1m\sum_{i=1}^m\Delta_k(X_i,Y_i)-R_k\right|
 >\frac12R_k+C_1\frac xm
 \ \middle|\ \mathcal H
 \right)
 \leq2e^{-x}.
\end{equation}
Take $x=\log(2K_0m)$ and use a union bound.  With conditional probability at
least $1-m^{-1}$, \eqref{eq:validation-concentration} holds for every $k$.
On this event, if $k_*$ minimizes $R_k$, then
\begin{align*}
 \frac12R_{\widehat k}-a
 &\leq \frac1m\sum_{i=1}^m\Delta_{\widehat k}(X_i,Y_i)\\
 &\leq \frac1m\sum_{i=1}^m\Delta_{k_*}(X_i,Y_i)
 \leq\frac32R_{k_*}+a,
\end{align*}
where $a=C_1\log(2K_0m)/m$.  Hence
$R_{\widehat k}\leq3R_{k_*}+4a$.  On the complementary event, $R_{\widehat k}\leq1$.  Taking conditional
expectations therefore gives
\[
 \Ep(R_{\widehat k}\mid\mathcal H)
 \leq3\min_{1\leq k\leq K_0}R_k
 +4a+\frac1m.
\]
Taking expectations and using
\[
 \Ep\min_{1\leq k\leq K_0}R_k
 \leq
 \min_{1\leq k\leq K_0}\Ep R_k
\]
proves \eqref{eq:validation-oracle} after increasing the constant.
\end{proof}

\subsection{Risk of the candidate estimators}

For $J\in\mathcal J_m$, let
\begin{equation}\label{eq:population-gram}
 G_J=\Ep_g\{\Phi_J(X)\Phi_J(X)^\top\}.
\end{equation}
For every $u\in\R^{D_J}$, orthonormality in $L^2(dx)$ and
\eqref{eq:G} give
\begin{equation}\label{eq:population-gram-bounds}
 c\norm u_{\ell^2}^2
 \leq u^\top G_Ju
 \leq C\norm u_{\ell^2}^2.
\end{equation}
Thus $cI\preceq G_J\preceq CI$.

\begin{lemma}\label{lem:candidate-risk}
There exist constants $a>0$ and $C_1<\infty$ such that, for every
$\beta\in[\beta_-,\beta_+]$, every $J\in\mathcal J_m$,
$f\in\F_\beta(M)$, and $g\in\G(c,C)$,
\begin{equation}\label{eq:candidate-risk}
 \Ep_{f,g}\normg{\widehat f_J-f}^2
 \leq
 C_1\left(
 \inf_{v\in V_J}\normg{v-f}^2
 +\frac{D_J}{m}
 +D_J e^{-am/D_J}
 \right).
\end{equation}
The constants are uniform over the displayed choices of
$\beta,J,f$, and $g$.
\end{lemma}

\begin{proof}
By \eqref{eq:app-kernel-bounds},
$\norm{\Phi_J(x)}_{\ell^2}^2=K_J(x,x)\leq A_\psi D_J$.
The matrix Chernoff inequality applied to the positive-semidefinite matrices
$\Phi_J(X_i)\Phi_J(X_i)^\top$, together with
\eqref{eq:population-gram-bounds}, gives
\begin{equation}\label{eq:gram-tail}
 \Pp\{\lambda_{\min}(\widehat G_J)<c/2\}
 \leq D_J\exp(-am/D_J)
\end{equation}
for a constant $a=a(c,A_\psi)>0$; see
\cite[Corollary~5.2 and Remark~5.3]{Tropp2012}.

Let $f_J^g=\Phi_J^\top\theta_J^g$ be the orthogonal projection of $f$ onto
$V_J$ in $L^2(g)$.  On the event
$\{\lambda_{\min}(\widehat G_J)\geq c/2\}$, the empirical normal equations
imply
\begin{equation}\label{eq:normal-equation}
 \widehat\theta_J-\theta_J^g
 =\widehat G_J^{-1}\frac1m\sum_{i\in I_0}
 \Phi_J(X_i)\{Y_i-f_J^g(X_i)\}.
\end{equation}
The summands have mean zero, since $f_J^g$ is the $L^2(g)$ projection and
$\Ep(Y\mid X)=f(X)$.  Moreover,
\begin{align}
 \Ep\norm{\frac1m\sum_{i\in I_0}
 \Phi_J(X_i)\{Y_i-f_J^g(X_i)\}}_{\ell^2}^2
 &=\frac1m\Ep\left[
 K_J(X,X)\{Y-f_J^g(X)\}^2\right]
 \notag\\
 &\lesssim\frac{D_J}{m}.
 \label{eq:score-second-moment}
\end{align}
Indeed, the regression identity, the projection property, and the fact that
$0\in V_J$ give
\[
 \Ep\{Y-f_J^g(X)\}^2
 =
 \Ep\{Y-f(X)\}^2+\normg{f-f_J^g}^2
 \leq\frac14+\normg f^2
 \leq\frac54.
\]
On the good Gram event,
$\norm{\widehat G_J^{-1}}_{\mathrm{op}}\leq2/c$. Therefore
\eqref{eq:normal-equation}, \eqref{eq:score-second-moment}, and
$G_J\preceq CI$ give
\begin{equation}\label{eq:estimation-part-candidate}
 \Ep\left[
 \normg{\Phi_J^\top\widehat\theta_J-f_J^g}^2
 \1_{\{\lambda_{\min}(\widehat G_J)\geq c/2\}}
 \right]
 \lesssim\frac{D_J}{m}.
\end{equation}
Let
\[
 \mathcal G_J=\{\lambda_{\min}(\widehat G_J)\geq c/2\}.
\]
Since $f_J^g$ is the $L^2(g)$ projection of $f$ onto $V_J$, we have
\[
 \normg{f-f_J^g}^2
 =
 \inf_{v\in V_J}\normg{f-v}^2,
\]
and $\Phi_J^\top\widehat\theta_J-f_J^g\in V_J$ is orthogonal in
$L^2(g)$ to $f_J^g-f$.  Pointwise clipping to $[0,1]$ cannot increase
squared error from $f\in[0,1]$.  Hence, on $\mathcal G_J$,
\[
 \normg{\widehat f_J-f}^2
 \leq
 \normg{\Phi_J^\top\widehat\theta_J-f_J^g}^2
 +\normg{f_J^g-f}^2.
\]
On $\mathcal G_J^c$, both $\widehat f_J$ and $f$ take values in $[0,1]$,
so $\normg{\widehat f_J-f}^2\leq1$.  Taking expectations and using
\eqref{eq:gram-tail} and \eqref{eq:estimation-part-candidate} proves
\eqref{eq:candidate-risk}.
\end{proof}

\begin{proof}[Proof of \Cref{prop:pilot}]
Conditional on $I_0$, apply \cref{lem:validation} to the candidates
$(\widehat f_J)_{J\in\mathcal J_m}$ and the validation block $I_1$.  Since
$|\mathcal J_m|=O(\log m)$, \eqref{eq:validation-oracle} and
\cref{lem:candidate-risk} imply
\begin{equation}\label{eq:pilot-oracle}
 \Ep\normg{\widetilde f-f}^2
 \lesssim
 \min_{J\in\mathcal J_m}
 \left\{
 \inf_{v\in V_J}\normg{v-f}^2+\frac{D_J}{m}
 +D_Je^{-am/D_J}
 \right\}
 +\frac{\log m}{m}.
\end{equation}
For $f\in\F_\beta(M)$,
\begin{equation}\label{eq:approximation-g}
 \inf_{v\in V_J}\normg{v-f}^2
 \leq\normg{P_Jf-f}^2
 \leq C A_B^2M^22^{-2J\beta}.
\end{equation}
For $\beta\in[\beta_-,\beta_+]$, put
\[
 D_\beta^*=m^{d/(2\beta+d)}.
\]
Since
\[
 D_\beta^*\leq m^{d/(2\beta_-+d)}
\]
and
\[
 \inf_{\beta\in[\beta_-,\beta_+]}D_\beta^*
 =
 m^{d/(2\beta_++d)}
 \longrightarrow\infty,
\]
there is, for all sufficiently large $m$ uniformly over
$\beta\in[\beta_-,\beta_+]$, an index $J\in\mathcal J_m$ such that
\begin{equation}\label{eq:optimal-J-pilot}
 D_J\asymp m^{d/(2\beta+d)}.
\end{equation}
At this index, the approximation and variance terms in
\eqref{eq:pilot-oracle} are both of order
$m^{-2\beta/(2\beta+d)}$, with constants uniform in $\beta$.

Furthermore, uniformly over $J\in\mathcal J_m$,
\begin{equation*}
 \frac{m}{D_J}\geq m^{2\beta_-/(2\beta_-+d)},
\end{equation*}
so the exponential term in \eqref{eq:pilot-oracle} is negligible.  Also,
\begin{equation*}
 \frac{2\beta}{2\beta+d}
 \leq\frac{2\beta_+}{2\beta_++d}<1,
\end{equation*}
which implies that $(\log m)/m=o(m^{-2\beta/(2\beta+d)})$ uniformly over the
fixed smoothness interval.  Thus
\begin{equation}\label{eq:pilot-risk-g}
 \sup_{\substack{f\in\F_\beta(M)\\g\in\G(c,C)}}
 \Ep\normg{\widetilde f-f}^2
 \lesssim m^{-2\beta/(2\beta+d)}.
\end{equation}
Finally, $c\normtwo h^2\leq\normg h^2$ and $m\asymp n$, which gives
\eqref{eq:pilot-risk}.  Measurability follows from the finite candidate list,
the specified first-minimizer rules, and the measurable matrix-inverse and
clipping operations.
\end{proof}

\section{Proof of the residual-bound lemmas}\label{sec:bound-proof}

\subsection{Deterministic coercivity}

\begin{proof}[Proof of \Cref{lem:coercivity}]
Write
\begin{equation}\label{eq:v-w-decomposition}
 v=P_Jr,
 \qquad
 w=(I-P_J)r.
\end{equation}
Projection contraction and $g\leq C$ give
\begin{equation}\label{eq:q-upper-proof}
 q_J(r;g)^{1/2}
 =\normtwo{P_J(rg)}
 \leq\normtwo{rg}
 \leq C\normtwo r,
\end{equation}
which proves \eqref{eq:q-upper-deterministic}.

If $v=0$, \eqref{eq:coercive-one} is immediate.  Otherwise, because
$v\in V_J$,
\begin{align}
 \normtwo{P_J(gv)}\normtwo v
 &\geq\ip{P_J(gv)}v
 =\int_\X g(x)v(x)^2\,dx
 \geq c\normtwo v^2.
 \label{eq:compressed-coercivity}
\end{align}
Thus $\normtwo{P_J(gv)}\geq c\normtwo v$.  On the other hand,
$\normtwo{P_J(gw)}\leq C\normtwo w$.  Since
\begin{equation*}
 P_J(rg)=P_J(gv)+P_J(gw),
\end{equation*}
the reverse triangle inequality gives \eqref{eq:coercive-one}.  Because
$\normtwo v\geq\normtwo r-\normtwo w$, the same inequality gives
\eqref{eq:coercive-total}.  Finally,
\begin{equation*}
 c\normtwo v\leq q_J(r;g)^{1/2}+C\normtwo w.
\end{equation*}
Squaring, using $(x+y)^2\leq2x^2+2y^2$, and then using $v\perp w$, yields
\begin{align*}
 \normtwo r^2
 &=\normtwo v^2+\normtwo w^2\\
 &\leq\frac{2}{c^2}q_J(r;g)
 +\left(1+\frac{2C^2}{c^2}\right)\normtwo w^2,
\end{align*}
which is \eqref{eq:coercive-square}.
\end{proof}

\subsection{Conditional moments and the upper bound}

\begin{proof}[Proof of \Cref{lem:cross-product-bound}]
Condition throughout on $\mathcal H$.  Since $a$ is
$\mathcal H$-measurable and $\mathcal H$ is independent of the observations
indexed by $A\cup B$, conditional on $\mathcal H$ the function $a$ is fixed,
while the observations in $A$ and $B$ remain independent with their original
law. 
Define
\begin{equation}\label{eq:W-mu-Sigma}
 W=\{Y-a(X)\}\Phi_J(X),
 \qquad
 \mu=\Ep(W\mid\mathcal H),
 \qquad
 \Sigma=\operatorname{Var}(W\mid\mathcal H).
\end{equation}
The coordinates of $\mu$ are the coefficients of $P_J((f-a)g)$ in the
level-$J$ scaling basis, so
\begin{equation}\label{eq:mu-q}
 \norm\mu_{\ell^2}^2=q_J(r;g).
\end{equation}
Let
\begin{equation*}
 e_H=S_{H,J}(a)-\mu,
 \qquad H\in\{A,B\}.
\end{equation*}
Conditional on $\mathcal H$, the vectors $e_A,e_B$ are independent and
centered, and each has conditional covariance $\Sigma/\ell$.  Expanding
\begin{equation*}
 T_J(a)=\ip{\mu+e_A}{\mu+e_B}
\end{equation*}
shows immediately that $\Ep(T_J(a)\mid \mathcal H)=\norm\mu_{\ell^2}^2$.  The three
centered terms $\mu^\top e_A$, $\mu^\top e_B$, and $e_A^\top e_B$ have zero
pairwise covariance.  Therefore
\begin{equation}\label{eq:exact-T-variance}
 \operatorname{Var}(T_J(a)\mid \mathcal H)
 =\frac{2}{\ell}\mu^\top\Sigma\mu
 +\frac1{\ell^2}\operatorname{tr}(\Sigma^2).
\end{equation}

For every $u\in\R^{D_J}$,
\begin{align*}
u^\top\Sigma u
&\leq\Ep\bigl(\{u^\top W\}^2\mid\mathcal H\bigr)\\
 &\leq(1+A_0)^2\int_\X\{u^\top\Phi_J(x)\}^2g(x)\,dx\\
 &\leq(1+A_0)^2C\norm u_{\ell^2}^2.
\end{align*}
Hence $\norm\Sigma_{\mathrm{op}}\leq(1+A_0)^2C$ and
$\operatorname{tr}(\Sigma^2)\leq D_J\norm\Sigma_{\mathrm{op}}^2$.
Combining these bounds with \eqref{eq:mu-q} and
\eqref{eq:exact-T-variance} proves \eqref{eq:T-moments}.

It remains to construct the upper bound.  Write $q=q_J(r;g)$ and take
$t=\lambda\sqrt{D_J}/\ell$.  Chebyshev's inequality and
\eqref{eq:T-moments} give
\begin{align}
 \Pp\{|T_J(a)-q|>q/2+t\mid \mathcal H\}
 &\leq
 V\frac{q/\ell+D_J/\ell^2}{(q/2+t)^2}
 \notag\\
 &\leq V\left(\frac1{2\lambda}+\frac1{\lambda^2}\right).
 \label{eq:Chebyshev-bound}
\end{align}
For the last inequality, use
$(q/2+t)^2\geq2qt$, $(q/2+t)^2\geq t^2$, and $D_J\geq1$.
Choose $\lambda=\lambda_\zeta$ so that the last expression is at most
$\zeta$.  On the complementary event,
\begin{equation*}
 T_J(a)+t\in[q/2,3q/2+2t].
\end{equation*}
Multiplying the positive part by two gives
\begin{equation*}
 q\leq2\{T_J(a)+t\}_+\leq3q+4t,
\end{equation*}
which is \eqref{eq:bound-event}.
\end{proof}

\section{Sharpness and design-density smoothness}\label{sec:sharpness}

The upper construction shows that separation by a sufficiently large constant
multiple of $\rho_n(s)$ is enough.  We now prove that a smaller order cannot
suffice.  The lower bound uses the fixed design density $g\equiv1$, so it is
independent of the nuisance issue.

\subsection{A testing lower bound}

\begin{theorem}\label{thm:testing-lower}
Fix $0<r<t<S$ and $M>0$.  There is a function
$f_0\in\bigcap_{0<u<S}\F_u(M)$ such that, for every positive sequence
$\delta_n=o\{\rho_n(r)\}$,
\begin{equation}\label{eq:testing-lower-statement}
 \inf_{\phi_n}
 \left[
 \Pp_{f_0,1}(\phi_n=1)
 +
 \sup_{\substack{f\in\F_r(M)\\
 \dist(f,\B_t(M))\geq\delta_n}}
 \Pp_{f,1}(\phi_n=0)
 \right]
 \longrightarrow1.
\end{equation}
The infimum is over all tests based on the $n$ observations, and the subscript
$1$ denotes the uniform design density.
\end{theorem}

\begin{proof}
Choose a constant $p_0\in(0,1/4)$ so small that
\begin{equation}\label{eq:p0-small}
 \norm{p_0}_{\mathbb B^u_{2,\infty}}\leq M/4
 \qquad\text{for every }u\in[0,S].
\end{equation}
This is possible by \eqref{eq:app-constant-Besov-norm}; we take
$f_0\equiv p_0$.

For every sufficiently large resolution level $j$, take the interior family
$(\psi_{j,k})_{k\in\mathcal K_j}$ from \eqref{eq:app-interior-family}.  By
\eqref{eq:app-interior-cardinality}, its cardinality
$m_j:=|\mathcal K_j|$ satisfies $m_j\asymp2^{jd}$; the remaining properties
used below are recorded at the end of \cref{app:wavelets}.
Choose $j=j_n$ so that
\begin{equation}\label{eq:lower-j-choice}
 2^{j_n}\asymp n^{2/(4r+d)}.
\end{equation}
Then
\begin{equation}\label{eq:rho-j-equivalence}
 2^{-j_nr}\asymp\rho_n(r),
 \qquad
 n^2 2^{-j_n(4r+d)}\asymp1.
\end{equation}

Set $\widetilde\rho_n=2^{-j_nr}$ and define
\begin{equation}\label{eq:kappa-lower}
 \kappa_n=
 \max\left\{
 \left(\frac{\delta_n}{\widetilde\rho_n}\right)^{1/2},
 2^{-j_n(t-r)/2}
 \right\}.
\end{equation}
Because $\delta_n=o\{\rho_n(r)\}$ and
$\widetilde\rho_n\asymp\rho_n(r)$, we have $\kappa_n\to0$.  Fix a  constant $a_0>0$ and put
\begin{equation}\label{eq:b-lower}
 b_n=a_0\kappa_n\frac{2^{-j_nr}}{\sqrt{m_{j_n}}}.
\end{equation}
For $\theta=(\theta_k)_{k\in\mathcal K_{j_n}}
\in\{-1,1\}^{m_{j_n}}$, define
\begin{equation}\label{eq:lower-alternatives}
 f_\theta
 =p_0+b_n\sum_{k\in\mathcal K_{j_n}}\theta_k\psi_{j_n,k}.
\end{equation}

We first check that these are admissible alternatives.  The only nonzero
detail block of $f_\theta-p_0$ is at level $j_n$, and
\begin{equation}\label{eq:r-Besov-size-lower}
 2^{j_nr}\normtwo{Q_{j_n}(f_\theta-p_0)}
 =2^{j_nr}b_n\sqrt{m_{j_n}}
 =a_0\kappa_n.
\end{equation}
Together with \eqref{eq:p0-small}, this places every $f_\theta$ in
$\B_r(M)$ for all sufficiently large $n$.  Bounded overlap and the standard
$L^\infty$ scaling of wavelets give
\begin{equation}\label{eq:supnorm-lower-perturbation}
 \norm{f_\theta-p_0}_\infty
 \lesssim b_n2^{j_nd/2}
 \lesssim a_0\kappa_n2^{-j_nr}=o(1),
\end{equation}
so $0\leq f_\theta\leq1$ for all large $n$.

For any $h\in\B_t(M)$, projection onto the level-$j_n$ detail space gives
\begin{align}
 \normtwo{f_\theta-h}
 &\geq\normtwo{Q_{j_n}(f_\theta-h)}
 \notag\\
 &\geq a_0\kappa_n2^{-j_nr}-M2^{-j_nt}.
 \label{eq:distance-smoother-lower}
\end{align}
The second term is negligible relative to the first because
$\kappa_n\geq2^{-j_n(t-r)/2}$:
\begin{equation*}
 \frac{M2^{-j_nt}}{a_0\kappa_n2^{-j_nr}}
 \leq\frac{M}{a_0}2^{-j_n(t-r)/2}\longrightarrow0.
\end{equation*}
Hence
\begin{equation}\label{eq:distance-lower-final}
 \dist(f_\theta,\B_t(M))
 \geq\frac{a_0}{2}\kappa_n\widetilde\rho_n
 \gg\delta_n,
\end{equation}
where the last relation follows from
$\kappa_n\geq(\delta_n/\widetilde\rho_n)^{1/2}$.  Thus every $f_\theta$ belongs
to the alternative set in \eqref{eq:testing-lower-statement} for all large
$n$.

It remains to show that the alternatives are statistically indistinguishable
from $f_0$.  Let $P_0$ be the one-observation law under $(f_0,g\equiv1)$ and
$P_\theta$ the law under $(f_\theta,g\equiv1)$.  Let
$L_\theta=dP_\theta^{\otimes n}/dP_0^{\otimes n}$, and let $\overline P_n$ be
the uniform mixture of $P_\theta^{\otimes n}$ over all sign vectors.  A direct
Bernoulli likelihood calculation gives, for independent sign vectors
$\theta,\theta'$,
\begin{equation}\label{eq:Bernoulli-cross-moment}
 \Ep_0(L_\theta L_{\theta'})
 =\left(
 1+\frac{b_n^2}{p_0(1-p_0)}
 \sum_{k\in\mathcal K_{j_n}}\theta_k\theta_k'
 \right)^n.
\end{equation}
Indeed, for one observation the cross moment equals
\begin{equation*}
 1+\frac1{p_0(1-p_0)}
 \int_\X(f_\theta-p_0)(f_{\theta'}-p_0)\,dx,
\end{equation*}
and orthonormality gives the sum in
\eqref{eq:Bernoulli-cross-moment}.

For large $n$, the expression inside parentheses is positive uniformly in the
signs, because $m_{j_n}b_n^2=o(1)$.  Set
\[
 \lambda_n=\frac{b_n^2}{p_0(1-p_0)}.
\]
Since the variables $\theta_k\theta_k'$ are independent Rademacher variables,
\eqref{eq:Bernoulli-cross-moment}, $(1+x)^n\leq e^{nx}$, and
$\cosh x\leq e^{x^2/2}$ give
\begin{align}
 1+\chi^2(\overline P_n,P_0^{\otimes n})
 &=
 \Ep_{\theta,\theta'}
 \left(
 1+\lambda_n
 \sum_{k\in\mathcal K_{j_n}}\theta_k\theta_k'
 \right)^n
 \notag\\
 &\leq
 \Ep_{\theta,\theta'}
 \exp\left\{
 n\lambda_n
 \sum_{k\in\mathcal K_{j_n}}\theta_k\theta_k'
 \right\}
 \notag\\
 &=
 \{\cosh(n\lambda_n)\}^{m_{j_n}}
 \notag\\
 &\leq
 \exp\left\{
 \frac12m_{j_n}n^2\lambda_n^2
 \right\}
 \notag\\
 &\leq
 \exp\left\{
 C a_0^4\kappa_n^4
 n^22^{-j_n(4r+d)}
 \right\}
 \longrightarrow1.
 \label{eq:chi-square-lower}
\end{align}
Here the last inequality follows from
$m_{j_n}\asymp2^{j_nd}$ and the definition of $b_n$, while the convergence
follows from \eqref{eq:rho-j-equivalence} and $\kappa_n\to0$.  Consequently,
\[
 \TV(\overline P_n,P_0^{\otimes n})
 \leq
 \frac12
 \sqrt{\chi^2(\overline P_n,P_0^{\otimes n})}
 \longrightarrow0.
\]  

For every test $\phi_n$,
\begin{equation*}
 \Pp_{f_0,1}(\phi_n=1)
 +\Ep_{\theta}\Pp_{f_\theta,1}(\phi_n=0)
 \geq1-\TV(\overline P_n,P_0^{\otimes n})=1-o(1).
\end{equation*}
The mixture-averaged type-II error is at most the supremum over the
alternative class, so the infimum in
\eqref{eq:testing-lower-statement} has limit inferior at least one.  The
constant test $\phi_n\equiv0$ has error sum exactly one, giving the reverse
bound and completing the proof.
\end{proof}

\subsection{Consequences for adaptive confidence sets}

\begin{theorem}\label{thm:separation-sharp}
Fix $0<r<t\leq s<S$ with $s>2r$, and let
$0<\alpha,\alpha'<1$ satisfy $2\alpha+\alpha'<1$.  Suppose
$a_n\geq0$ and $a_n=o\{\rho_n(r)\}$.  There is no sequence of confidence sets $C_n$ that is
honest at level $1-\alpha$ under the uniform design over
\begin{equation}\label{eq:sharpness-parameter-space}
 \F_s(M)
 \cup
 \left\{f\in\F_r(M):
 \dist(f,\B_t(M))\geq a_n\right\}
\end{equation}
and for which there is a constant $K<\infty$ satisfying
\begin{equation}\label{eq:sharpness-small-diameter}
 \limsup_{n\to\infty}
 \sup_{f\in\F_s(M)}
 \Pp_{f,1}\left\{\diam(C_n)>K\eps_n(s)\right\}
 \leq\alpha'.
\end{equation}
\end{theorem}

\begin{proof}
Assume such confidence sets exist.  Because $s>2r$,
\begin{equation}\label{eq:eps-smaller-rho}
 \eps_n(s)=o\{\rho_n(r)\}.
\end{equation}
We may suppose that $K > 1$. Set
\begin{equation}\label{eq:delta-reduction}
 b_n=\max\{a_n,K\eps_n(s)\},
 \qquad
 \delta_n=\{\rho_n(r)b_n\}^{1/2}.
\end{equation}
Then
\begin{equation}\label{eq:delta-relations}
 a_n=o(\delta_n),
 \qquad
 K\eps_n(s)=o(\delta_n),
 \qquad
 \delta_n=o\{\rho_n(r)\}.
\end{equation}
Let $f_0$ be the smooth constant from \cref{thm:testing-lower}.  Replacing
$C_n$ by its closure can only improve coverage and does not change its
diameter.  We may therefore assume that $C_n$ is a random closed subset of
$L^2(dx)$.  Under the standing random-set measurability convention, define
\begin{equation}\label{eq:confidence-to-test}
 \phi_n=
 \1\left\{
 C_n\cap
 \left\{h\in L^2(dx):
 \normtwo{h-f_0}>\frac{\delta_n}{2}\right\}
 \neq\varnothing
 \right\}.
\end{equation}
Under $f_0$, if $f_0\in C_n$ and $\phi_n=1$, then
$\diam(C_n)>\delta_n/2$.   
By honesty,
\eqref{eq:sharpness-small-diameter}, and
\eqref{eq:delta-relations},
\begin{equation}\label{eq:reduction-type-one}
 \Pp_{f_0,1}(\phi_n=1)
 \leq\alpha+\alpha'+o(1).
\end{equation}
Now let $f\in\F_r(M)$ satisfy
$\dist(f,\B_t(M))\geq\delta_n$.  Since
$f_0\in\B_t(M)$, we have $\normtwo{f-f_0}\geq\delta_n$.  If $f\in C_n$, then
$\phi_n=1$, so honesty gives
\begin{equation}\label{eq:reduction-type-two}
 \sup_{\substack{f\in\F_r(M)\\
 \dist(f,\B_t(M))\geq\delta_n}}
 \Pp_{f,1}(\phi_n=0)
 \leq\alpha+o(1).
\end{equation}
The alternatives in \eqref{eq:reduction-type-two} are contained in the
parameter space \eqref{eq:sharpness-parameter-space} for all large $n$ because
$\delta_n\gg a_n$.  Combining
\eqref{eq:reduction-type-one} and \eqref{eq:reduction-type-two} gives a testing
error sum at most $2\alpha+\alpha'+o(1)<1$, contradicting
\cref{thm:testing-lower}.
\end{proof}

\begin{remark}\label{rem:shell-maximality}
Fix confidence levels satisfying $2\alpha+\alpha'<1$.  Apply
\cref{thm:separation-sharp} to the $j$th shell with
$r=s_j$, $t=2s_j=s_{j+1}$, and $s=\beta_+$.  Since
$\beta_+>s_N\geq2s_j$ for every $j<N$, the condition $s>2r$ holds.
Moreover,
\[
 \F_{s_{j+1}}(M)\subseteq\B_{s_{j+1}}(M),
 \qquad
 \dist\bigl(f,\B_{s_{j+1}}(M)\bigr)
 \leq
 \dist\bigl(f,\F_{s_{j+1}}(M)\bigr).
\]
Thus the alternative class in \cref{thm:separation-sharp} is contained in
the corresponding separated shell.  It follows that the deletion radius in
\eqref{eq:separated-space} cannot be replaced by
$o\{\rho_n(s_j)\}$ while retaining honesty and the
$\eps_n(\beta_+)$ diameter on the top smooth class.  In this sense, the
separated parameter space is rate-sharp up to multiplicative constants.
\end{remark}

\subsection{No design smoothness required}

For $\gamma>0$, write $B^\gamma_{2,\infty}([0,1]^d)$ for the standard Besov
space defined using any boundary-wavelet basis of regularity larger than
$\gamma$.  At smoothness zero, we use the fixed basis and normalization in
\eqref{eq:Besov-norm}.  The next proposition shows that the overlap class is
not merely a class of piecewise smooth densities.

\begin{proposition}\label{prop:rough-density}
Assume $c<1<C$.  There is a density $g\in\G(c,C)$ such that
$g\notin B^\gamma_{2,\infty}([0,1]^d)$ for every $\gamma>0$.
\end{proposition}

\begin{proof}
Fix $j_1$ sufficiently large and, for every $j\geq j_1$, take the interior
family $(\psi_{j,k})_{k\in\mathcal K_j}$ from
\eqref{eq:app-interior-family}.  Put
$m_j=|\mathcal K_j|\asymp2^{jd}$, as in
\eqref{eq:app-interior-cardinality}, and
\begin{equation}\label{eq:rough-h}
 H_j=\frac{j^{-2}}{\sqrt{m_j}}
 \sum_{k\in\mathcal K_j}\psi_{j,k},
 \qquad
 h=\sum_{j\geq j_1}H_j.
\end{equation}
Orthonormality gives $\normtwo{H_j}=j^{-2}$ and $\int H_j=0$.  Compact support,
bounded overlap, and $\norm{\psi_{j,k}}_\infty\lesssim2^{jd/2}$ give
$\norm{H_j}_\infty\lesssim j^{-2}$.  Hence the series converges uniformly to a
bounded mean-zero function.  Moreover, $Q_jh=H_j$, and therefore, for every
$0<s<S$,
\begin{equation}\label{eq:rough-detail-divergence}
 \sup_{j\geq j_1}2^{js}\normtwo{Q_jh}
 =\sup_{j\geq j_1}2^{js}j^{-2}=\infty.
\end{equation}
The boundary-wavelet characterization of Besov spaces on bounded domains
\cite{DeVoreKyriazisWang1998,gine2021mathematical} gives
$h\notin B^s_{2,\infty}([0,1]^d)$ for every $0<s<S$. Given any $\gamma>0$,
choose $0<s<\min\{\gamma,S\}$.  The embedding
$B^\gamma_{2,\infty}\subset B^s_{2,\infty}$ then shows that
$h\notin B^\gamma_{2,\infty}$.

Since $h$ is bounded and has mean zero, choose $\eta>0$ sufficiently small
that
\begin{equation*}
 \eta\norm{h}_\infty\leq\min\{1-c,C-1\},
\end{equation*}
and set $g=1+\eta h$.  Then $\int g=1$ and $c\leq g\leq C$.  Adding the
constant one cannot restore smoothness: constants belong to every
$B^\gamma_{2,\infty}$, and membership of $g$ would imply
$h=(g-1)/\eta\in B^\gamma_{2,\infty}$, a contradiction.
\end{proof}

\begin{proof}[Proof of \Cref{cor:threshold}]
The procedure in \cref{thm:main} is uniform over the full class $\G(c,C)$ and
uses neither a design-density smoothness index nor a design-density Besov
radius.  When $c<1<C$, the final claim
follows from \cref{prop:rough-density}.
\end{proof}

\subsection{Some design coverage is necessary}

The uniform lower bound on $g$ is qualitatively different from a smoothness
assumption: it guarantees that every region of positive Lebesgue volume is
observed with comparable frequency.  Some such coverage condition is
indispensable for the unweighted loss, as the following elementary
identifiability obstruction shows.

\begin{proposition}
\label{prop:vanishing-design}
Suppose a fixed design density $g$ vanishes on a nonempty open set
$U\subset[0,1]^d$.  Fix $0<s<S$, $M>0$, and $0<\alpha<1/2$.  If random sets
$C_n$ satisfy
\begin{equation}\label{eq:vanishing-honesty}
 \liminf_{n\to\infty}\inf_{f\in\F_s(M)}
 \Pp_{f,g}(f\in C_n)\geq1-\alpha,
\end{equation}
then there are $f_0\in\F_s(M)$ and $\delta>0$ such that
\begin{equation}\label{eq:vanishing-diameter}
 \liminf_{n\to\infty}
 \Pp_{f_0,g}\{\diam(C_n)\geq\delta\}
 \geq1-2\alpha.
\end{equation}
In particular, uniformly shrinking $L^2(dx)$ confidence sets are impossible.
\end{proposition}

\begin{proof}
Choose a nonnegative, nonzero function $b\in C_c^\infty(U)$ and a constant function 
$f_0$ such that
\[
 0<f_0<\frac12,
 \qquad
 \norm{f_0}_{\mathbb B^s_{2,\infty}}\leq\frac M2.
\]
Choose $u>0$ satisfying
\[
 u\leq
 \min\left\{
 \frac{M}{2\norm{b}_{\mathbb B^s_{2,\infty}}},
 \frac{1-f_0}{\norm b_\infty}
 \right\},
\]
and set
\[
 f_1=f_0+ub.
\]
Then $0\leq f_1\leq1$ and
\[
 \norm{f_1}_{\mathbb B^s_{2,\infty}}
 \leq
 \norm{f_0}_{\mathbb B^s_{2,\infty}}
 +u\norm{b}_{\mathbb B^s_{2,\infty}}
 \leq M.
\]
Thus $f_0,f_1\in\F_s(M)$. The laws under $(f_0,g)$ and $(f_1,g)$ are
identical, because the two regression functions differ only where $g=0$.
Under this common law, \eqref{eq:vanishing-honesty} and the union bound imply
that both $f_0$ and $f_1$ belong to $C_n$ with limiting probability at least
$1-2\alpha$.  On that event,
$\diam(C_n)\geq\normtwo{f_1-f_0}$.  Taking
$\delta=\normtwo{f_1-f_0}>0$ proves \eqref{eq:vanishing-diameter}.
\end{proof}

\appendix

\crefalias{section}{appendix}

\section{Wavelet background and notation}\label{app:wavelets}

Fix  $S>\beta_+$.  We use a compactly supported,
boundary-corrected orthonormal wavelet basis on $
 \X=[0,1]^d$ 
whose regularity and number of vanishing moments are both greater than  $S$.  The basis is obtained by tensorizing the interval construction of
Cohen, Daubechies, and Vial \cite{CohenDaubechiesVial1993}; see also
\cite[Section~4.3.5]{gine2021mathematical}.  We describe the resulting
functions, spaces, and projections at the level needed in this paper.  The
explicit formulas used to construct the finitely many boundary functions are
not needed.

\subsection{One-dimensional interval wavelets}

Begin with a compactly supported orthonormal Daubechies wavelet system on
$\mathbb R$.  It consists of a scaling function $\varphi$ and an associated
mother wavelet $\psi$.  The scaling functions describe approximation at a
given resolution, while the wavelets describe the additional detail revealed
when the resolution is increased.  For integers $j,k$, define
\begin{equation}\label{eq:app-whole-line-wavelets}
 \varphi^{\mathbb R}_{j,k}(x)
 =
 2^{j/2}\varphi(2^jx-k),
 \qquad
 \psi^{\mathbb R}_{j,k}(x)
 =
 2^{j/2}\psi(2^jx-k).
\end{equation}
The index $j$ is the resolution level: the functions in
\eqref{eq:app-whole-line-wavelets} are supported on intervals of length
comparable to $2^{-j}$.  The index $k$ specifies their location, and the
factor $2^{j/2}$ gives each function unit $L^2(\mathbb R)$ norm.

Choose the whole-line system to have H\"older--Zygmund regularity
$R_{\mathrm w}>S$ and an integer number $N_{\mathrm w}>S$ of vanishing
moments:
\begin{equation}\label{eq:app-whole-line-moments}
 \int_{\mathbb R}x^m\psi(x)\,dx=0,
 \qquad
 m=0,\ldots,N_{\mathrm w}-1.
\end{equation}
These are the regularity and moment assumptions needed for the Besov
characterization stated below.

The ordinary whole-line functions whose supports lie inside $(0,1)$ may be
used unchanged on the interval.  Those whose supports cross $0$ or $1$ cannot
simply be truncated: truncation would break orthonormality, the vanishing
moments of the wavelets, and the ability of the scaling spaces to reproduce
low-degree polynomials.  The Cohen--Daubechies--Vial construction replaces
these functions by finitely many boundary-adapted scaling functions and
wavelets.

More precisely, after relabeling, there is an integer $J_0\geq0$ such that,
for every $j\geq J_0$, the construction provides orthonormal families
\[
 \bigl(\varphi^I_{j,k}\bigr)_{k\in\mathcal I_j},
 \qquad
 \bigl(\psi^I_{j,k}\bigr)_{k\in\mathcal I_j}
\]
in $L^2([0,1])$, where $|\mathcal I_j|=2^j$.  Each family consists of
finitely many left- and right-boundary functions, whose number is independent
of $j$ and depends only on the chosen wavelet system and its boundary
correction, together with the ordinary whole-line dilates and translates
whose supports lie in $(0,1)$.  After $L^2$ normalization, the boundary
functions are level-$j$ dilates of fixed finite collections of prototypes.
We choose $J_0$ sufficiently large that the left- and right-boundary families
have disjoint supports at every level $j\geq J_0$.  This choice depends only
on the fixed whole-line wavelet system and its CDV boundary correction.

Define the one-dimensional approximation and detail spaces by
\begin{equation}\label{eq:app-one-dimensional-spaces}
 V_j^I
 =
 \operatorname{span}
 \{\varphi^I_{j,k}:k\in\mathcal I_j\},
 \qquad
 W_j^I
 =
 \operatorname{span}
 \{\psi^I_{j,k}:k\in\mathcal I_j\}.
\end{equation}
The displayed families are orthonormal bases of their respective spaces, and
\begin{equation}\label{eq:app-one-dimensional-mra}
 V_{j+1}^I=V_j^I\oplus W_j^I,
 \qquad
 L^2([0,1])
 =
 V_{J_0}^I\oplus\bigoplus_{j\geq J_0}W_j^I.
\end{equation}
All direct sums in this appendix are orthogonal.

For every $j\geq J_0$ and every polynomial $p$ of degree at most
$N_{\mathrm w}-1$, the restriction $p|_{[0,1]}$ belongs to $V_j^I$;
equivalently, the orthogonal projection onto $V_j^I$ reproduces
$p|_{[0,1]}$ exactly.  Since
$W_j^I\perp V_j^I$, every interval wavelet has the corresponding vanishing
moments:
\begin{equation}\label{eq:app-interval-moments}
 \int_0^1x^m\psi^I_{j,k}(x)\,dx=0,
 \qquad
 m=0,\ldots,N_{\mathrm w}-1.
\end{equation}
In particular, constants belong to every $V_j^I$, and every interval wavelet
has mean zero.  The boundary-adapted functions retain the same scale of
compact support and the same uniform localization properties as the ordinary
interior functions.  These properties are part of the
Cohen--Daubechies--Vial interval construction; see
\cite[Section~4]{CohenDaubechiesVial1993}.

\subsection{Tensor-product wavelets}

We obtain a basis on $\X=[0,1]^d$ by tensorization.  Set
\begin{equation}\label{eq:app-index-sets}
 \Lambda_j=\mathcal I_j^d,
 \qquad
 \mathcal E
 =
 \{0,1\}^d\setminus\{(0,\ldots,0)\}.
\end{equation}
For
\[
 \lambda=(k_1,\ldots,k_d)\in\Lambda_j,
\]
define the level-$j$ tensor-product scaling function
\begin{equation}\label{eq:app-tensor-scaling}
 \varphi_{j,\lambda}(x)
 =
 \prod_{\ell=1}^d\varphi^I_{j,k_\ell}(x_\ell),
 \qquad
 x=(x_1,\ldots,x_d)\in\X.
\end{equation}

For $e\in\{0,1\}$, write
\[
 u^0_{j,k}=\varphi^I_{j,k},
 \qquad
 u^1_{j,k}=\psi^I_{j,k}.
\]
For a tensor type
\[
 \mathbf e=(e_1,\ldots,e_d)\in\mathcal E,
\]
define
\begin{equation}\label{eq:app-tensor-wavelets}
 \psi^{\mathbf e}_{j,\lambda}(x)
 =
 \prod_{\ell=1}^d u^{e_\ell}_{j,k_\ell}(x_\ell).
\end{equation}
Thus $e_\ell=0$ selects a scaling function in coordinate $\ell$, while
$e_\ell=1$ selects a wavelet.  The excluded all-zero type gives the scaling
function in \eqref{eq:app-tensor-scaling}.  Every
$\mathbf e\in\mathcal E$ contains at least one wavelet factor and therefore
represents detail in at least one coordinate.

A tensor-product function may be boundary adapted in any coordinate whose
one-dimensional factor lies near an endpoint.  If every factor is an
ordinary interior function, then the tensor product agrees with an ordinary
whole-line tensor-product wavelet whose support lies in the interior of
$\X$.

Define the multivariate approximation and detail spaces by
\begin{equation}\label{eq:app-multivariate-spaces}
 V_j
 =
 \operatorname{span}
 \{\varphi_{j,\lambda}:\lambda\in\Lambda_j\},
 \qquad
 W_j
 =
 \operatorname{span}
 \left\{
 \psi^{\mathbf e}_{j,\lambda}:
 \mathbf e\in\mathcal E,\ \lambda\in\Lambda_j
 \right\}.
\end{equation}
Tensorizing \eqref{eq:app-one-dimensional-mra} gives
\begin{equation}\label{eq:app-multivariate-mra}
 V_{j+1}=V_j\oplus W_j,
 \qquad
 L^2(\X)
 =
 V_{J_0}\oplus\bigoplus_{j\geq J_0}W_j.
\end{equation}
Since $|\mathcal I_J|=2^J$,
\begin{equation}\label{eq:app-dimension}
 D_J:=\dim(V_J)=|\Lambda_J|=2^{Jd}.
\end{equation}

\subsection{Projections and the Besov decomposition}

We use lower-case $j$ for a running resolution level and upper-case $J$ for
a particular truncation level.  Let $P_J$ denote the orthogonal projection
onto $V_J$, and let $Q_j$ denote the orthogonal projection onto $W_j$.
Then
\begin{equation}\label{eq:app-PJ}
 P_Jh
 =
 \sum_{\lambda\in\Lambda_J}
 \ip{h}{\varphi_{J,\lambda}}\varphi_{J,\lambda}
\end{equation}
and
\begin{equation}\label{eq:app-QJ}
 Q_jh
 =
 \sum_{\mathbf e\in\mathcal E}
 \sum_{\lambda\in\Lambda_j}
 \ip{h}{\psi^{\mathbf e}_{j,\lambda}}
 \psi^{\mathbf e}_{j,\lambda}.
\end{equation}
By \eqref{eq:app-multivariate-mra},
\[
 Q_j=P_{j+1}-P_j,
\]
and every $h\in L^2(\X)$ has the expansion
\begin{equation}\label{eq:app-wavelet-expansion}
 h=P_{J_0}h+\sum_{j\geq J_0}Q_jh
\end{equation}
with convergence in $L^2(\X)$.  Parseval's identity gives
\begin{equation}\label{eq:app-detail-energy}
 \normtwo{Q_jh}^2
 =
 \sum_{\mathbf e\in\mathcal E}
 \sum_{\lambda\in\Lambda_j}
 \left|
 \ip{h}{\psi^{\mathbf e}_{j,\lambda}}
 \right|^2.
\end{equation}

The quantity $\normtwo{Q_jh}^2$ is therefore the total $L^2$ energy of the
wavelet coefficients at spatial scale $2^{-j}$.  The condition
\[
 2^{js}\normtwo{Q_jh}\lesssim1
\]
says that (the square root of) this detail energy decays at least as fast as $2^{-js}$.  This is
the multiscale interpretation of the Besov norm in
\eqref{eq:Besov-norm}.  For every $0<s<S$, that norm is equivalent to the
standard $B^s_{2,\infty}(\X)$ norm; see
\cite{DeVoreKyriazisWang1998}.  At $s=0$, we use
\eqref{eq:Besov-norm} as the definition of the endpoint space.

For later use, define the vector of level-$J$ scaling functions and the
associated projection kernel by
\begin{equation}\label{eq:app-Phi-kernel}
 \Phi_J(x)
 =
 \bigl(\varphi_{J,\lambda}(x)\bigr)_{\lambda\in\Lambda_J}
 \in\mathbb R^{D_J},
 \qquad
 K_J(x,z)
 =
 \Phi_J(x)^\top\Phi_J(z).
\end{equation}
Equivalently,
\[
 K_J(x,z)
 =
 \sum_{\lambda\in\Lambda_J}
 \varphi_{J,\lambda}(x)\varphi_{J,\lambda}(z),
\]
and $K_J$ is the integral kernel of $P_J$:
\begin{equation}\label{eq:app-kernel-representation}
 P_Jh(x)=\int_\X K_J(x,z)h(z)\,dz.
\end{equation}

\subsection{Properties used in the proofs}

\begin{proposition}\label{prop:app-wavelet-estimates}
There is a constant $A_\psi<\infty$, depending only on the fixed
one-dimensional wavelet system, its boundary correction, and $d$, with the
following properties.  All implicit constants below have the same permitted
dependence.

\begin{enumerate}[label=\textup{(\roman*)}]
\item\label{item:app-localization}
For every $j\geq J_0$, every $\lambda\in\Lambda_j$, and every
$\mathbf e\in\mathcal E$,
\[
 \operatorname{diam}
 \bigl(\operatorname{supp}\varphi_{j,\lambda}\bigr)
 \leq A_\psi2^{-j},
 \qquad
 \operatorname{diam}
 \bigl(\operatorname{supp}\psi^{\mathbf e}_{j,\lambda}\bigr)
 \leq A_\psi2^{-j},
\]
and
\[
 \norm{\varphi_{j,\lambda}}_\infty
 \leq A_\psi2^{jd/2},
 \qquad
 \norm{\varphi_{j,\lambda}}_1
 \leq A_\psi2^{-jd/2},
\]
\[
 \norm{\psi^{\mathbf e}_{j,\lambda}}_\infty
 \leq A_\psi2^{jd/2},
 \qquad
 \norm{\psi^{\mathbf e}_{j,\lambda}}_1
 \leq A_\psi2^{-jd/2}.
\]
At each fixed level, every point of $\X$ belongs to the supports of at most
$A_\psi$ scaling functions and, across all tensor types, at most $A_\psi$
detail functions.

\item\label{item:app-kernel}
For every $J\geq J_0$,
\begin{equation}\label{eq:app-kernel-bounds}
 \sup_{x\in\X}K_J(x,x)\leq A_\psi D_J,
 \qquad
 \sup_{x\in\X}\int_\X|K_J(x,z)|\,dz\leq A_\psi.
\end{equation}
Consequently,
\begin{equation}\label{eq:app-Linfty-stability}
 \norm{P_Jh}_\infty\leq A_\psi\norm h_\infty.
\end{equation}

\item\label{item:app-nesting}
If $0\leq s\leq t<S$, then
\begin{equation}\label{eq:app-nesting}
 \norm{h}_{\mathbb B^s_{2,\infty}}
 \leq
 \norm{h}_{\mathbb B^t_{2,\infty}}.
\end{equation}
Consequently,
\begin{equation}\label{eq:app-class-nesting}
 \B_t(M)\subseteq\B_s(M),
 \qquad
 \F_t(M)\subseteq\F_s(M).
\end{equation}

\item\label{item:app-tail}
If $s\geq\beta_->0$, $h\in\B_s(M)$, and $J\geq J_0$, then
\begin{equation}\label{eq:app-wavelet-tail}
 \normtwo{(I-P_J)h}^2
 \leq
 \frac{M^2}{1-2^{-2\beta_-}}\,2^{-2Js}.
\end{equation}

\item\label{item:app-constants}
Constant functions belong to $V_{J_0}$.  Hence, if $b$ is constant, then
\[
 P_{J_0}b=b,
 \qquad
 Q_jb=0,\quad j\geq J_0,
\]
and, since $|\X|=1$,
\begin{equation}\label{eq:app-constant-Besov-norm}
 \norm{b}_{\mathbb B^s_{2,\infty}}
 =
 2^{J_0s}|b|.
\end{equation}

\item\label{item:app-interior}
Fix the tensor type
\[
 \mathbf e_\circ=(1,0,\ldots,0)\in\mathcal E.
\]
For every sufficiently large $j$, let $\mathcal K_j$ index the
tensor-product detail functions of type $\mathbf e_\circ$ for which every
one-dimensional factor is an ordinary, unmodified interior function.  Relabel this family as
\begin{equation}\label{eq:app-interior-family}
 \bigl(\psi_{j,k}\bigr)_{k\in\mathcal K_j}.
\end{equation}
It is an orthonormal family in $W_j$.  There is a constant
$N_{\mathrm{ov}}<\infty$, independent of $j$, such that every point of
$\X$ belongs to the supports of at most $N_{\mathrm{ov}}$ members of the
family, and 
\begin{equation}\label{eq:app-interior-cardinality}
 m_j:=|\mathcal K_j|\asymp2^{jd}.
\end{equation}
Moreover,
\begin{equation}\label{eq:app-interior-mean-zero}
 \int_\X\psi_{j,k}(x)\,dx=0
\end{equation}
and
\begin{equation}\label{eq:app-interior-norms}
 \norm{\psi_{j,k}}_\infty\lesssim2^{jd/2},
 \qquad
 \norm{\psi_{j,k}}_1\lesssim2^{-jd/2},
\end{equation}
uniformly in $j$ and $k$.
\end{enumerate}
\end{proposition}

\begin{proof}
The Cohen--Daubechies--Vial construction supplies the one-dimensional
orthonormality, stable boundary structure, localization bounds, bounded
overlap, reproduction of constants by the scaling spaces, and zero mean of
the wavelets stated above.  Tensorization gives the multivariate orthonormal
basis and the decomposition \eqref{eq:app-multivariate-mra}.  It also gives
the localization and norm bounds in part~\ref{item:app-localization}: the
one-dimensional bounds multiply across coordinates, and only finitely many
tensor types occur.  The bounded-overlap property is likewise preserved,
with a constant depending only on $d$.

For fixed $x$, only a bounded number of the values
$\varphi_{J,\lambda}(x)$ are nonzero.  Part~\ref{item:app-localization} and
\eqref{eq:app-dimension} therefore give
\[
 K_J(x,x)
 =
 \sum_{\lambda\in\Lambda_J}
 |\varphi_{J,\lambda}(x)|^2
 \lesssim
 2^{Jd}
 =
 D_J.
\]
Similarly,
\[
\begin{aligned}
 \int_\X|K_J(x,z)|\,dz
 &\leq
 \sum_{\lambda:\,\varphi_{J,\lambda}(x)\neq0}
 |\varphi_{J,\lambda}(x)|
 \norm{\varphi_{J,\lambda}}_1 \\
 &\lesssim
 2^{Jd/2}2^{-Jd/2}
 \lesssim1.
\end{aligned}
\]
The kernel representation \eqref{eq:app-kernel-representation} then yields
\eqref{eq:app-Linfty-stability}.

If $0\leq s\leq t$, then
\[
 2^{J_0s}\leq2^{J_0t},
 \qquad
 2^{js}\leq2^{jt},
 \quad j\geq J_0.
\]
The definition \eqref{eq:Besov-norm} therefore gives
\eqref{eq:app-nesting}, and the class inclusions in
\eqref{eq:app-class-nesting} follow.

For $h\in\B_s(M)$, the orthogonal decomposition
\eqref{eq:app-wavelet-expansion} gives
\[
 \normtwo{(I-P_J)h}^2
 =
 \sum_{j\geq J}\normtwo{Q_jh}^2.
\]
The definition of the Besov ball gives
\[
 \normtwo{Q_jh}\leq M2^{-js}.
\]
Consequently,
\[
\begin{aligned}
 \normtwo{(I-P_J)h}^2
 &\leq
 M^2\sum_{j\geq J}2^{-2js} \\
 &=
 \frac{M^2}{1-2^{-2s}}\,2^{-2Js} \\
 &\leq
 \frac{M^2}{1-2^{-2\beta_-}}\,2^{-2Js},
\end{aligned}
\]
which proves \eqref{eq:app-wavelet-tail}.

Since constants belong to every one-dimensional scaling space $V_j^I$,
tensorization places constants in every $V_j$.  Part~\ref{item:app-constants}
then follows from \eqref{eq:Besov-norm}.

Finally, at level $j$, only a fixed number of one-dimensional functions near
each endpoint are boundary adapted or have support meeting an endpoint.
Thus, in each coordinate, $2^j-O(1)$ indices give ordinary interior
functions.  For the fixed tensor type $\mathbf e_\circ$, the number of
multi-indices for which every factor is interior is therefore
\[
 \bigl(2^j-O(1)\bigr)^d
 =
 2^{jd}-O\bigl(2^{j(d-1)}\bigr)
 \asymp2^{jd}.
\]
This proves \eqref{eq:app-interior-cardinality}.  The selected functions are
members of the orthonormal basis of $W_j$, so they are orthonormal and inherit
the bounded-overlap property.  Since
$\mathbf e_\circ=(1,0,\ldots,0)$, each tensor product contains a wavelet
factor in its first coordinate.  That factor has integral zero, and Fubini's
theorem gives \eqref{eq:app-interior-mean-zero}.  The bounds in
\eqref{eq:app-interior-norms} follow from
part~\ref{item:app-localization}.
\end{proof}

\bibliographystyle{amsplain}
{\small
\bibliography{adaptive}
}

\end{document}